\documentclass[12pt]{article}
\usepackage[dvipsnames,usenames]{color}
\usepackage{amsmath}
\usepackage{amsfonts}
\usepackage{amssymb}
\usepackage{mathrsfs}
\usepackage{amsthm}

\newcommand{\ov}{\overline}
\newcommand{\p}{\partial}
\newcommand{\w}{\widetilde}

\newtheorem{thm}{Theorem}[section]
\newtheorem{cor}[thm]{Corollary}
\newtheorem{lem}[thm]{Lemma}
\newtheorem{prop}[thm]{Proposition}
\newtheorem{example}[thm]{Example}

\newtheorem{conj}[thm]{Conjecture}
\numberwithin{equation}{section}

\medskip

\begin{document}
\title{\bf Regular multi-types   and the Bloom conjecture}
\author{Xiaojun Huang\footnote{
Supported in part by NSF-1665412} \ \  and\  Wanke Yin\footnote{Supported in part by NSFC-11722110 and NSFC-11571260 }} \maketitle

\abstract{ We prove   equality of the vector field (iterated
commutator) type and the regular contact type, which together with
the Bloom theorem on equality of the Levi-form type and the regular
contact type provides a complete solution of a long standing open
problem of Bloom (\cite{Bl2}) in the case of complex dimension
three. For general dimensions, we verify the Bloom conjecture when
$s=n-2$, which provides the first positive result in the
pseudoconvexity sensitive case for a real hypersurface in ${\mathbb
C}^n$ after his important work   in 1981 (\cite{Bl2}).}

\section{Introduction}
Let $D$ be a smoothly bounded pseudoconvex domain in  ${\mathbb
C}^n$ for $n\ge 2$. Many analytic and geometric properties of  $D$
are  determined by its boundary holomorphic invariants. To
generalize his subelliptic estimate for the $\ov\partial$-Neumann
problem from bounded strongly pseudoconvex domains \cite{FK} to
bounded weakly pseudoconvex domains in ${\mathbb C}^2$, Kohn in a
fundamental paper \cite{Kohn1} investigated three different boundary
invariants for $D\subset {\mathbb C}^2$.  These invariants describe,
respectively,  the maximum order of contact with smooth holomorphic
curves at  a boundary point, degeneracy of the Levi-form along the
CR directions and the length  of the iterated Lie brackets of
boundary CR vector fields as well as their conjugates  needed to
recover the boundary contact direction. Kohn proved that all these
invariants are in fact the same, called the type value of  a point
on $\partial D\subset {\mathbb C}^2$. When this type value is finite
at each point, Kohn's work in \cite{Kohn1} together with that of
Greiner \cite{Gr} (see also Rothschild-Stein \cite{RS}) gives the
precise information of how much the subelliptic gain one obtains for
the $\ov\partial$-Neumann problem for a smoothly bounded weakly
pseudoconvex domain in ${\mathbb C}^2$. The finite type condition
initiated by  the work of Kohn has played fundamental roles in late
studies of many problems. For instance, Bedford-Fornaess \cite{BF}
(see also the later work of Fornaess-Sibony \cite{FS}) exploited
peak functions over weakly pseudoconvex domains of finite type in
${\mathbb C}^2$ and discovered  close connections of the type value
of the boundary and the H\"older-continuity of the peak functions up
to the boundary.

Generalization of Kohn's notion of the boundary finite type
condition to higher dimensions has been a subject under  extensive
investigations in the past  40 years in Several Complex Variables.
Kohn later introduced a  finite type condition in higher dimensions
through the subelliptic multiplier ideals  \cite{Kohn2}. The
understanding of this type  has later revived to be a very active
field of studies through the work of many people. (See
 Basyrov-Nicoara-Zaitsev\cite{BNZ},
Diederich-Fornaess \cite{DF}, Siu \cite{Siu1},
 Kim-Zaistev \cite{KZ}, Zaistev
\cite{Zai}  and the reference therein.)  Bloom \cite{Bl1} and
Bloom-Graham \cite{BG1} established Kohn's original notion of types
in ${\mathbb C}^2$ to any dimensions which are called the regular
multi-types. D'Angelo \cite{DA1} introduced his important and famous
notion of (D'Angelo) finite type conditions by considering the order
of contact with not just smooth complex manifolds but possibly
singular complex analytic varieties, which turns out to be
equivalent to the existence of the  subelliptic estimate by the work
of Kohn \cite{Kohn2}, Diederich-Fornaess \cite{DF} and Catlin
\cite{Cat2}. Catlin in \cite{Cat1} studied his famous multitype
condition as well as its connection with the boundary stratification
in terms of the degeneracy of  Levi forms. McNeal \cite{Mc} and
later Boas-Straube \cite{BS} studied the the line type condition for
convex domains and proved its equivalence with the D'Angelo type,
which was further applied by Fu-Isav-Krantz [FIK] to prove the
equivalence of the D'Angelo type with the regular contact type for
Reinhardt domains.

All these type conditions  mentioned above were introduced through
different aspects of studies. Revealing the connections among them
always brings our deeper understanding of the subject. For instance,
proving that the Kohn multiplier ideal  type is equivalent to the
finite D'Angelo type  would provide a new and much more direct
solution of the $\ov\partial$-Neumann problem \cite{Cat2}.

In this paper, we will be concerned with  the three multi-regular
types. We will be especially interested in the question when all
these types are equivalent, known as the Bloom problem. We will show
that the vector field type (which is also called the H\"{o}rmander
type in some other contents) coincides with the regular contact type
in the case of dimension three. This result, together with the work
of Bloom in 1981 \cite{Bl2} on the equality of the Levi-form type
with the contact type in ${\mathbb C}^3$, provides a complete
solution of Bloom's conjecture in the case of dimension three. In
general dimensions, we will show that the first three
pseudoconvexity-sensitive $(n-2)$-types all are the same. Our paper
makes a progress along the lines of the long-standing  Bloom
conjecture after his striking work in 1981.


\section{Statement of the main theorem}

Let $M\subset \mathbb{C}^n$ be a smooth real hypersurface with $p\in
M$. Then  dim$_{\mathbb{C}}T^{1,0}_pM=n-1$ for  $p\in M$. For any
$1\leq s\leq n-1$, we have the following three sets of important
local holomorphic  invariants (\cite{Bl2}), used to describe the
finite holomorphic  non-degeneracy of $M$ at $p$.

\medskip
\noindent (i): The $s$-contact type $a^{(s)}(M,p)$:
 \begin{equation}\begin{split}
  a^{(s)}(M,p)=\sup\limits_{X}\big\{r|\ &\exists  \text{ an $s$-dimensional complex submanifold}\ X\\
  &\text{whose order of contact  with $M$ at $p$ is $r$}\big\}.
 \end{split} \end{equation}

Let $\rho$ be a defining function of $M$ near $p$, namely, $\rho\in
C^\infty(U)$ with $U$ an open neighbourhood of $p\in\mathbb{C}^n$
and   $U\cap M=\{\rho=0\}\cap U$, $d\rho|_{U\cap M}\neq 0$.  Remark
that the order of contact of $X$ with $M$ at $p$ is defined as the
order of vanishing of $\rho|_X$ at $p$.
\bigskip

\noindent (ii) The $s$-vector field type $t^{(s)}(M,p)$:
\medskip

Let $B$ be an $s$-dimensional subbundle of $T^{1,0}M$. We let
$\mathcal{M}_1(B)$ be the $C^\infty(M)$-module spanned by the smooth
tangential $(1,0)$ vector fields $L$ with $L|_q\in B|_q$ for each
$q\in M$, together with the conjugate of these vector fields.

For $\mu\geq 1$, we let $\mathcal{M}_\mu(B)$ denote the
$C^\infty(M)$-module spanned by commutators of length less than or
equal to $\mu$ of vector fields from  $\mathcal{M}_1(B)$. A
commutator of length $\mu$ of vector fields in $\mathcal{M}_1(B)$ is
a vector field of the following form:
$[Y_{\mu},[Y_{\mu-1},\cdots,[Y_2,Y_1]\cdots]$. Here $Y_j\in
\mathcal{M}_1(B)$. Define $t^{(s)}(B,p)=m$ if $\langle F,\p
\rho\rangle(p)=0$ for any $F\in \mathcal{M}_{m-1}(B)$ but $\langle
G,\p \rho\rangle(p)\neq 0$ for a certain $G\in \mathcal{M}_{m}(B)$.
Then

\begin{equation}\begin{split}
  t^{(s)}(M,p)=\sup\limits_{B}\{t(B,p)|\ B\ \text{is an $s$-dimensional subbundle of\ }\ T^{1,0}M\}.
 \end{split} \end{equation}

$t^{(s)}(B,p)$ is the smallest length of the commutators by vector
fields  in $\mathcal{M}_1(B)$ to recover the complex contact
direction  in $\mathbb{C}T_pM$. $t^{(s)}(M,p)$ is the largest
possible value among all  $t^{(s)}(B,p)'s$. Namely, $t^{(s)}(M,p)$
describes the degeneracy of the most degenerate $s$-subbundles of
$T^{1,0}M$. Notice that it is intrinsically defined, independent of
the ambient embedded space.

\bigskip
\noindent (iii) The $s$-Levi form tpype $c^{(s)}(M,p)$:
\medskip

Let $B$ be as in (ii). Let $\mathcal{L}_{M,p}$ be a Levi form
associated with a defining function $\rho$ near $p$ of $M$. For
$V_B=\{L_1,\cdots,L_s\}$, a basis of smooth sections  of $B$ near
$p$, we define the trace of $\mathcal{L}_{M,p}$ along $V_B$ by

\begin{equation}
  \text{tr}_{V_B}\mathcal{L}_{M,p}=\sum\limits_{j=1}^{s}\langle  [L_j,\ov{L_j}], \p \rho\rangle (p).
 \end{equation}

We define $c(B,p)=m$ if for any $m-3$ vector fields
$F_1,\cdots,F_{m-3}$ of $\mathcal{M}_1(B)$ and any basis  $V_B$, it
holds that
$$
F_{m-3}\cdots F_{1}\big(\text{tr}_{V_B}\mathcal{L}_{M,p}\big)(p)=0
$$
and for a certain choice of $m-2$ vector fields $G_1,\cdots,G_{m-2}$
of  $\mathcal{M}_1(B)$ and a certain basis  $V_B$, we have
$$
G_{m-2}\cdots G_{1}\big(\text{tr}_{V_B}\mathcal{L}_{M,p}\big)(p)\neq
0.
$$
Then
\begin{equation}\begin{split}
  c^{(s)}(M,p)=\sup\limits_{B}\{c(B,p): \  B\ \text{ is an $s$-dimensional subbundle of} \ T^{1,0}M
  \}.
 \end{split} \end{equation}

In his fundamental paper \cite{Kohn1}, when $n=2$, Kohn showed that
$t^{(1)}(M,p)=c^{(1)}(M,p)=a^{(1)}(M,p)$. Bloom-Graham \cite{BG2}
and Bloom \cite{Bl1} proved   that
$$t^{(n-1)}(M,p)=c^{(n-1)}(M,p)=a^{(n-1)}(M,p)\ \ \hbox{for}\  M\subset
{\mathbb C}^n.$$ And for any $1\leq s\leq n-2$, Bloom in \cite{Bl2}
observed that $a^{(s)}(M,p)\leq c^{(s)}(M,p)$ and $a^{(s)}(M,p)\leq
t^{(s)}(M,p)$. For  these results to hold there is no need  to
assume  the pseudoconvexity  of $M$. However, the following example
of Bloom shows that for $n\geq  3$, when $M$ is not pesudoconvex, it
may happen that $a^{(s)}(M,p)< c^{(s)}(M,p)$ and $a^{(s)}(M,p)<
t^{(s)}(M,p)$ for $1\leq s\leq n-2$.

\begin{example}[Bloom, \cite{Bl2}] Let $\rho=2\text{Re}(w)+(z_2+\ov{z_2}+|z_1|^2)^2$
and let $M=\{(z_1,z_2,w)\in \mathbb{C}^3|\ \rho=0\}$. Let $p=0$.
Then $a^{(1)}(M,p)=4$ but $c^{(1)}(M,p)=t^{(1)}(M,p)=\infty$ .
\end{example}

With the pseudoconvexity assumption of  $M$, Bloom in \cite{Bl2}
showed that when $M\subset {\mathbb C}^3,$
$a^{(1)}(M,p)=c^{(1)}(M,p)$. Motivated by this result, Bloom in 1981
\cite{Bl2} formulated the following conjecture:

\begin{conj}\label {bloom-conj} Let $M\subset \mathbb{C}^n$ be a pseudoconvex real hypersurface with $n\geq 3$. Then for any $1\leq s\leq n-2$ and $p\in M$,
\begin{equation*}\begin{split}
t^{(s)}(M,p)=c^{(s)}(M,p)=a^{(s)}(M,p).
\end{split} \end{equation*}

\end{conj}

In this paper, we make a progress along the lines of the Bloom conjecture by  presenting the proof of the following theorem:

\begin{thm}\label{mainthm}
Let $M\subset \mathbb{C}^n$ be a smooth pseudoconvex real
hypersurface with $n\geq 3$. Then for  $s= n-2$ and any $p\in M$, it
holds that
\begin{equation*}\begin{split}
t^{(n-2)}(M,p)=a^{(n-2)}(M,p)=c^{(n-2)}(M,p).
\end{split} \end{equation*}

\end{thm}

In particular,  we obtain a proof of the remaining case of the Bloom
conjecture in the case of complex dimension three by showing that
when $M\subset {\mathbb C}^3$, we  have $t^{(1)}(M,p)=a^{(1)}(M,p).$
($s=n-2=1$ for $n=3$). This, together with the work of Bloom in 1981
on the equality $c^{(1)}(M,p)=a^{(1)}(M,p)$ for $M\subset {\mathbb
C}^3$, finally provides a complete solution of the Bloom conjecture
in the three dimensional case.

\begin{thm}\label{thm-dim-three} The Bloom conjecture holds in the case of complex dimension three. Namely, for   a smooth pseudoconvex real hypersurface $M\subset \mathbb{C}^3$  and $p\in M$, it holds that
\begin{equation*}\begin{split}
t^{(1)}(M,p)=a^{(1)}(M,p)=c^{(1)}(M,p).
\end{split} \end{equation*}

\end{thm}


Our proof of Theorem \ref{mainthm} is a combination of analytic and
geometric  arguments along the lines of CR geometry. (See  the book
by Baouendi-Ebenfelt-Rothschild \cite {BER}). Our paper is mainly on
commutators of vector fields. Commutators of vector fields
 are not just important in complex analysis but also play
a fundamental role in many problems bordering complex analysis and
sub-elliptic analysis. For instance, in the paper of Adwan-Berhanu
\cite{AB}, the commutator type condition of  vector fields is
 crucially applied to get various real analytic hypo-ellipticity
results. See also the book of Berhanu-Cordaro-Hounie \cite{BCH} and
a paper of Derridj \cite{Derr} for many references and historical
discussions on this matter.
 In $\S 3$, we give a
general set-up and provide a normalization of the related vector
fields. In $\S 4$, we give a proof of Theorem \ref{mainthm} assuming
Theorem \ref{main-tech}. $\S5$ and $\S 6$ are dedicated to the long
proof of Theorem \ref{main-tech} on a sort of  uniqueness of a
complex linear PDE associated with a CR singular submanifold
contained in a psuedoconvex hypersurface.
Already from the work of Chern-Moser [CM], it is clear that a good
weight system is always important to single out the boundary
holomorphic invariants for real hypersurfaces  in a complex
Euclidean space. (See also
\cite{Bl2,BG2,Cat1,GS1,GS2,MW,HY,Kol1,Kol2} and the references
therein concerning different weight systems used in different
settings). In this work, we will adapt the weight system introduced
by Bloom in \cite{Bl2} to truncate the real hypersurface so that the
singular Frobenius-Nagano theorem can be applied. Then we will
derive contradictions if the theorem fails to be true by proving the
non-existence of certain CR manifolds (through Proposition
\ref{disc}) and CR singular manifolds (through Theorem
\ref{main-tech}) generated by the truncated CR vector fields in the
truncated hypersurface.  To attack the Bloom conjecture, it is
crucial to find a good use of the pseudoconvexity. In the case
considered by Bloom \cite{Bl2} for equality of $a^1(M,p)=c^1(M,p)$
of dimension three, the pseudoconvexity is  used to fundamentally
apply a neat result of Diederich-Fornaess [DF] (see also Freedman
[Fre]), which says the Lie-bracket operation is closed for sections
in the null space of Levi- form. This result  can not be applied in
the case we are considering. To handle the major difficulties in our
consideration, the pseudoconvexity is used for the validity of the
Hopf lemma (Proposition \ref{disc}) and for obtaining the triviality
of solutions of a complex linear equation with real part
plurisubharmonic (Theorem \ref{main-tech}). It is also interesting
to notice  the  important role played by the Euler vector field in
the course of  our proof.

To finishing off this section, we would like to  mention a result by
D'Angelo \cite{DA2} which shows that for a smooth pseudoconvex real
hypersurface $M\subset \mathbb{C}^n$ $(n\geq3)$ and for $p\in M$, if
one of the two invariants $t^{(1)}(B_1,p)\ \hbox{and}\
c^{(1)}(B_1,p)$ is $4$ then they both are $4$. Here $B_1$ is a one
dimensional smooth subbundle of $T^{(1,0)}M$.  He also obtained some
estimate of $c^{(1)}(B_1,p)$ in terms of  $t^{(1)}(B_1,p)$. Here we
mention that due to the pseudoconvexity, when $M$ at $p$ is not
strongly pseudoconvex, then $t^{(1)}(B_1,p)$ and $c^{(1)}(B_1,p)$
are at least $4$. We also mention the paper by McNeal-Mernik
(\cite{MM}) and the paper by D'Angelo \cite {DA4} on equality of the
regular contact type with the D'Angelo type when either one  is 4
under the pseudo convexity or even some weaker conditions.

\section{Normalization of CR vector fields}

Denote by $(z_1,\cdots,z_{n-1},w)$ the coordinates in
$\mathbb{C}^{n}$. Let $M\subset U$ be a smooth real hypersurface in
$ {\mathbb C}^n$ with $p\in M$ and let $\rho$ be a defining function
of $M$ near $p$. After a holomorphic change of coordinates, we may
assume that $p=0$ and $\rho$ takes the following form:
\begin{equation}\label{rho}
\begin{split}
\rho=-2\text{Re}(w)+\chi(z,\ov{z},\text{Im}{w}), \
\chi(z,\ov{z},\text{Im}{w})=O(|z^2|+|z\text{Im}{w}|).
\end{split}
\end{equation}
We will assume that $a^{(n-2)}(M,0)<\infty$ in all that follows, for
otherwise $$t^{(n-2)}(M,0), c^{(n-2)}(M,0)\ge
a^{(n-2)}(M,0)=\infty$$ and thus all these invariants coincide.
After a holomorphic change of coordinates of the form
$(z',w')=(z,w+O(2))$, we  assume that
\begin{equation}\label{R}
\begin{split}
\chi(z,0,0)=O(a^{(n-2)}(M,0)+1).
\end{split}
\end{equation}
Shrinking $U$ if necessary, we  assume $\frac{\p \rho}{\p
w}(z,w)\neq 0$ for $(z,w)\in U$. For a defining function $\rho$
defined over $U$ as in (\ref{rho}), write
\begin{equation}\begin{split}\label{L}
L_i=\frac{\p }{\p z_i}-\frac{\p \rho}{\p z_i}\big(\frac{\p \rho}{\p
w}\big)^{-1}\frac{\p }{\p w}\ \ \text{for}\ i=1,\cdots,n-1.
\end{split}\end{equation}
Then $\{L_i\}_{i=1}^{n-1}$ forms a basis for the space of  CR vector
fields along $M$. Let $B$ be an $(n-2)$ dimensional subbundle of
$T^{1,0}M$. Assume that the sections of $B$ are generated by a
certain linearly independent smooth CR vector fields $S_1,\cdots,
S_{n-2}$ along $M$ near $0$. After a linear holomorphic change of
coordinates, we assume that $S_j(0)=L_j(0)=\frac{\partial }{\p
z_j}|_0$ for $1\leq j \leq n-2$. Write
\begin{equation} \label{norm-001}
 S_j=\sum_{h=1}^{n-1}
a_{jh}L_h\ \hbox{with }a_{jh}(0)=\delta_{jh} \ \hbox{for }1\leq
j,h\leq n-2.
\end{equation}


\medskip

We recall the following fact from \cite[Lemma 5.2]{Bl2}, which gives
the transformation law for $\{L_j',\rho'\}$ and $\{L_j,\rho\}$ under
a holomorphic change of coordinates $(z',w')=F(z,w)$ with
$\rho=\rho'\circ F$, $F(0)=0$.

\begin{lem}\label{tranrel}
   Let $(z',w')=F(z,w)=(z_1',\cdots,z_{n-1}',w')$ be a new holomorphic coordinate system where
   $z'_j=z'_j(z_1,\cdots,z_{n-1})$ for $j=1,\cdots,n-1$, $w'=w$ with $z'(0)=0$. Then we have
\begin{equation}\begin{split}
F_{*}(L_i)=\sum\limits_{j=1}^{n-1}\frac{\p z_j'}{\p z_i}L_j' \ \
\text{for}\ i=1,\cdots,n-1.
\end{split}\end{equation}

\end{lem}
With $S_j$ and the frame $\{L_j\}$  being  given as above, we define
\begin{equation}\begin{split}
l^*_0=&\min_{1\leq j\leq n-2}\{k_j:\ k_j =\text{vanishing order of}\
a_{j(n-1)}(z_1,\cdots,z_{n-2},0,\ov{z_1}. \cdots,\ov{z_{n-2}},0)\
\hbox{at 0}\}.
\end{split}\end{equation}
In this section, for a smooth function $A$,  we write
$A^{(\tau)}(z,\ov{z})$ for the sum of  monomials of (ordinary)
degree $\tau$ in its Taylor expansion at $0$; also when we mention a
holomorphic change of coordinates, we  refer to a special type of
holomorphic maps of the form $(z',w')=F(z,w)$ as in Lemma
\ref{tranrel}.

\begin{lem}\label{firstn} Suppose $l_0^*\not =\infty$.
  After a holomorphic change of coordinates, we  have
   $$a^{(l^*_0)}_{j(n-1)}(0,\cdots,0,z_j,\cdots,z_{n-2},0,\cdots,0)=0\ \text{for all }\ 1\leq j\leq n-2.$$
\end{lem}
\begin{proof}
Let
$$
z_j'=z_j\ \text{for}\ 1\leq j\leq n-2,\ z_{n-1}'=z_{n-1}-\int_0^{z_{1}}a^{(l^*_0)}_{1(n-1)}(\xi,z_2,\cdots,z_{n-2},0,\cdots,0)d\xi,\ w=w'.
$$
Then in the new coordinates $(z',w')$, we have
\begin{equation}\begin{split}
&\frac{\p }{\p z_1}=\frac{\p }{\p z_1'}-a^{(l^*_0)}_{1(n-1)}(z_1,\cdots,z_{n-2},0,\cdots,0)\frac{\p }{\p z_{n-1}'},\\
&\frac{\p }{\p z_j}=\frac{\p }{\p z_j'}+O(l^*_0)\frac{\p }{\p z_{n-1}'}\ \text{for}\ 2\leq j\leq n-2,\\
& \frac{\p }{\p z_{n-1}}=\frac{\p }{\p z_{n-1}'}.
\end{split}\end{equation}
In the new coordinates, by Lemma \ref{tranrel}, we have
 \begin{equation}\begin{split}
 S_1=&\sum\limits_{h=1}^{n-1} a_{1h}L_h=a_{11}(L_1'-a^{(l^*_0)}_{1(n-1)}(z_1,\cdots,z_{n-2},0,\cdots,0)L_{n-1}')\\
   &+\sum_{h=2}^{n-2} a_{1h}(L_h'+O(l^*_0)L_{n-1}')+a_{1(n-1)}L_{n-1}'.
\end{split}\end{equation}
  Hence in the new coordinates, the coefficient $a_{1(n-1)}$ is changed to $$-a_{11}a^{(l^*_0)}_{1(n-1)}(z_1,\cdots,z_{n-2},0,\cdots,0)+\sum_{h=2}^{n-2}a_{1h}\cdot O(l^*_0)+a_{1(n-1)}.$$
Recall that $a_{1j}=\delta_{1j}+o(1)$ for $1\leq j\leq n-2$. Hence
in these new coordinates, which are still denoted by $(z,w)$, we
have $a^{(l^*_0)}_{1(n-1)}(z_1,\cdots,z_{n-2},0,\cdots,0)=0$.

Suppose that we  have achieved
$a^{(l^*_0)}_{h(n-1)}(0,\cdots,0,z_h,\cdots,z_{n-2},0,\cdots,0)=0$
for $1\leq h\leq j-1$.
 We next show that we can make $a^{(l^*_0)}_{j(n-1)}(0,\cdots,0,z_j,\cdots,z_{n-2},0,\cdots,0)=0$ after a
 holomorphic change of coordinates. Set $
w=w'$ and
$$
z_j'=z_j, 1\leq j\leq n-2,\
z_{n-1}'=z_{n-1}-\int_0^{z_j}a^{(l^*_0)}_{j(n-1)}(0,\cdots,0,\xi,z_{j+1},\cdots,z_{n-2},0,\cdots,0)d\xi.
$$
By a similar argument as in the proof for $a^{(l^*_0)}_{1(n-1)}(z_1,\cdots,z_{n-2},0,\cdots,0)=0$, we have
$$a^{(l^*_0)}_{j(n-1)}(0,\cdots,0,z_j,\cdots,z_{n-2},0,\cdots,0)=0.$$
Notice that this transformation of coordinates preserves the property:
 $$a^{(l^*_0)}_{h(n-1)}(0,\cdots,0,z_h,\cdots,z_{n-2},0,\cdots,0)=0 \hbox{ for }1\leq h\leq j-1.$$
 By induction, this completes the proof of Lemma \ref{firstn}.
\end{proof}

Next, after the normalization as in (\ref{firstn}),  we either have
\begin{equation}l_0:=\label{l0big} \min_{1\leq j\leq n-2} \{k_j:\ k_j=
\text{ord}_{z=0}\ a_{j(n-1)}(z_1,\cdots,z_{n-2},0,\ov{z_1},\\
\cdots,\ov{z_{n-2}},0)\}\geq a^{(n-2)}(M,0)
\end{equation}
 or
 $l_0\leq a^{(n-2)}(M,0)-1.$
In the case of (\ref{l0big}), we re-define $l_0$ to be
$a^{(n-2)}(M,0)$.
\begin{prop}\label{propn}
 Assume that $l_0\leq a^{(n-2)}(M,0)-1$. After a holomorphic change of coordinates we can
 normalize the coefficients of $\{S_j\}$
  to further  satisfy one of the following
two  normalization properties:
\begin{enumerate}
  \item[ (I)]  $a_{j(n-1)}^{(l_0)}(z_1,\cdots,z_{n-2},0,\ov{z_1},\cdots,\ov{z_{n-2}},0)$
  is holomorphic in $z_1,\cdots,z_{n-2}$ \hbox{ for each}\ j, and
  there exits
  $ j_0\in [2,n-2]$\ such that $a^{(l_0)}_{j(n-1)}(z_1,\cdots,z_{n-2},0,\ov{z_1},\cdots,\ov{z_{n-2}},0)=0$
  \ \hbox{for}\  $1\leq j\leq j_0-1,$
  $a^{(l_0)}_{j_0(n-1)}(0,\cdots,0,z_{j_0},\cdots,z_{n-2},0,\cdots,0)=0$,
  but
      \\ $a^{(l_0)}_{j_0(n-1)}(z_1,\cdots,z_{n-2},0,\cdots,0)\not\equiv 0$.\\

  \item[(II)]   $a^{(l_0)}_{1(n-1)}(z_1,\cdots,z_{n-2},0,\ov{z_1},\cdots,\ov{z_{n-2}},0)$ is not a holomorphic polynomial \\ and $a^{(l_0)}_{1(n-1)}(z_1,\cdots,z_{n-2},0,\cdots,0)=0$.
\end{enumerate}
\end{prop}

\begin{proof}
(I):   First, we assume that each
$a^{(l_0)}_{j(n-1)}(z_1,\cdots,z_{n-2},0,\ov{z_1},\cdots,\ov{z_{n-2}},0)$
is holomorphic, and each
  $a^{(l_0)}_{j(n-1)}$ satisfies the properties as in Lemma \ref{firstn}. Then
  $$a^{(l_0)}_{1(n-1)}(z_1,\cdots,z_{n-2},0,\cdots,0)=0.$$ By the definition of $l_0$, we can find the
   smallest $j_0\in [2,n-2]$ such that $$a^{(l_0)}_{j(n-1)}(z_1,\cdots,z_{n-2},0,\cdots,0)\equiv0$$ for all $1\leq j\leq  j_0-1$,
   but $a^{(l_0)}_{j_0(n-1)}(z_1,\cdots,z_{n-2},0,\cdots,0)\not\equiv0$. By Lemma \ref{firstn},
   this $j_0$ satisfies the property in
    part (I) of the proposition.

(II):  Next, assume that
$a^{(l_0)}_{j(n-1)}(z_1,\cdots,z_{n-2},0,\ov{z_1},\cdots,\ov{z_{n-2}},0)$
is not holomorphic for a certain $j\in [1,n-2]$. Switching $j$ with
the index $1$ and repeating the proof in Lemma \ref{firstn},
   we can make $a^{(l_0)}_{1(n-1)}(z_1,\cdots,z_{n-2},0,\cdots,0)=0$ and achieve the other normalization
   properties as in Lemma \ref{firstn}.  Notice that $l_0$ is not changed after this
   normalization procedure.  This completes the proof of the
   proposition.
\end{proof}

 Define the weight of $z_{j}$ and $\ov{z_{j}}$ for $1\leq j\leq n-2$  to be $1$.
 The weight of $z_{n-1}$ and $\ov{z_{n-1}}$ is defined to  be $l_0+1$ and  the weight of $w$ is defined to  be $m$
 that is the lowest weighted vanishing order of $\rho$ in the expansion of $\rho(z,0,\ov{z},0)$ at $0$  with respect to the
 weights  of $\{z_1,
 \cdots,z_{n-2},z_{n-1}\}$ just defined. In what follows, for a smooth function $A$,  we write $A^{[\sigma]}(z_1,\cdots,z_{n-1},\ov{z_1},\cdots,\ov{z_{n-1}})$
 for the weighted homogeneous part of weighted degree $\sigma$ with the weight system just defined in its Taylor expansion
 at $0$.
  Then we have  the following
\begin{prop}\label{propn2}
    In the case of Proposition \ref{propn} (II), we can further apply a holomorphic transformation of coordinates and
    change the basis $\{S_j\}$ if needed to make   the coefficients of
    $\{S_j\}$ in the expansion with respect to $\{L_j\}$
    satisfy
     one of the following two normalizations:
\begin{enumerate}
  \item[(1)]  $a^{(l_0)}_{1(n-1)}(z_1,0,\cdots,0,\ov{z_1},0,\cdots,0)\not \equiv 0$,\
  $a^{(l_0)}_{1(n-1)}(z_1,0,\cdots,0)=0$ \\($a^{(l_0)}_{1(n-1)}(z_1,\cdots,z_{n-2},0)=0$, \hbox{in fact
  }),
  $\rho^{[m]}(z_1,0,\cdots,0,z_{n-1},0,\ov{z_1},0,\cdots,0,\ov{z_{n-1}},0)$ is not identically zero
   (and contains no non-trivial holomorphic terms).

  \item[(2)] 
   For a certain $j\in [1,n-2]$, $a^{(l_0)}_{j(n-1)}(z_1,\cdots,z_{n-2},0,\ov{z_1},\cdots,\ov{z_{n-2}},0) $ is not holomorphic,\\
  $\sum_{k=1}^{n-2}z_ka^{(l_0)}_{k(n-1)}(z_1,\cdots,z_{n-2},0,\ov{z_1},\cdots,\ov{z_{n-2}},0)=0$,\\
   $\rho^{[m]}(z_1,\cdots,z_{n-1},0,\ov{z_1},\cdots,\ov{z_{n-1}},0)$ is not
   identically  zero  (and contains no non-trivial holomorphic terms).

\end{enumerate}
 \end{prop}

\begin{proof}

Consider the following change of coordinates:
 \begin{equation}\begin{split}
z_1'=z_1,\ z_j'=z_j-\alpha_jz_1,\ \text{for}\ 2\leq j\leq n-2,\
z_{n-1}'=z_{n-1},\ w'=w.
\end{split}\end{equation}
We first give a sufficient condition under which, for a generic
choice of  $\alpha_j$ with  $2\leq j\leq n-2$, we have
\begin{equation}\begin{split}\label{rhoan}
&\rho^{[m]}(z_1,0,\cdots,0,z_{n-1},0,\ov{z_1},0,\cdots,0,\ov{z_{n-1}},0)\not \equiv 0,\\
&a^{(l_0)}_{1(n-1)}(z_1,0,\cdots,0,\ov{z_1},0,\cdots,0)\
\text{contains no non-trivial holomorphic terms}.
\end{split}\end{equation}

Notice that
\begin{equation*}\label{norm-002}\begin{split}
&\rho^{[m]}(z_1,\cdots,z_{n-1},0,\ov{z_1},\cdots,\ov{z_{n-1}},0)\\
=&\rho^{[m]}(z_1',z_2'+\alpha_2z_1',\cdots,z_{n-2}'+\alpha_{n-2}z_{1}',z'_{n-1},0,\ov{z_1'},\ov{z_2'}+\ov{\alpha_2}\ov{z_1'},
\cdots,\ov{z_{n-2}'}+\ov{\alpha_{n-2}}\ov{z_{1}'},\ov{z'_{n-1}},0).
\end{split}\end{equation*}

 The coefficient of ${z'}_1^{t}{z'}_{n-1}^{\mu}\ov{{z'}_1}^{s}\ov{{z'}_{n-1}}^{\nu}$  with $t+s+\mu+\nu=m$ in its Taylor expansion is
$$
\sum_{\sum h_\lambda=t,\sum j_{\lambda}=s \atop{
H=(h_1,\cdots,h_{n-2}),J=(j_1,\cdots,j_{n-2})}}\rho^{[m]}_{(H\mu0)(J\nu0)}\alpha^{H}\ov{\alpha}^{J}.
$$
Here  $\alpha_1=1$, $\alpha=(\alpha_1,\cdots,\alpha_{n-2})$, and
$\rho^{[m]}_{(H\mu0)(J\nu0)}$ is the coefficient of
$$z_1^{h_1}\cdots z_{n-2}^{h_{n-2}}z_{n-1}^{\mu}\ov{z_1^{j_1}}\cdots
\ov{z_{n-2}^{j_{n-2}}}\ov{z_{n-1}^{\nu}}$$ in the Taylor expansion
of $\rho$ at $0$.

Notice that this term is $0$ for a generic choice of $\alpha$ if and
only if $\rho^{[m]}_{(H\mu0)(J\nu0)}=0$ for any pair $(H,J)$ with
$\sum h_\lambda=t,\ \sum j_{\lambda}=s$. By our choice of the weight
$m$, their exists a pair $(H\mu0)(J\nu0)$ with $|J|+\nu>0$ such that
$\rho^{[m]}_{(H\mu0)(J\nu0)}\neq 0$. Thus for a generic choice of
$\alpha_j's$, we have
$$\rho^{'[m]}(z_1',0,\cdots,0,z_{n-1}',0,\ov{z_1'},0,\cdots,0,\ov{z_{n-1}'},0)\not
\equiv 0.$$ Since $\rho^{[m]} $ contains no holomorphic terms, so is
$\rho'^{[m]}(z_1',0,\cdots,0,z_{n-1}',0,\ov{z_1'},0,\cdots,0,\ov{z_{n-1}'},0)$.
Hence for a generic choice of $\alpha$, the statement in the first
line of (\ref{rhoan}) holds.

Next notice that
\begin{equation}\begin{split}
\frac{\p}{\p z_1}=\frac{\p}{\p z_1'}-\sum_{\lambda=2}^{n-2}
\alpha_{\lambda}\frac{\p}{\p z_\lambda'},\ \frac{\p}{\p
z_j}=\frac{\p}{\p z_j'}\ \text{for}\ 2\leq j\leq n-1.
\end{split}\end{equation}
Thus by Lemma \ref{tranrel}, we know
\begin{equation}\begin{split}
L_1=L_1'-\sum_{\lambda=2}^{n-2} \alpha_{\lambda}L_\lambda',\
L_j=L_j'\  \text{for}\ 2\leq j\leq n-1.
\end{split}\end{equation}
Set
$S_1'=S_1+\sum_{\lambda=2}^{n-2} \alpha_{\lambda}S_\lambda$, $S_j'=S_j$. Then
\begin{equation}\begin{split}
S_1'=&\sum_{h=1}^{n-1} a_{1h}L_h+\sum_{\lambda=2}^{n-2} \alpha_{\lambda}\sum_{h=1}^{n-1}a_{\lambda h}L_h\\
=&\big(a_{11}+\sum_{\lambda=2}^{n-2}\alpha_\lambda a_{\lambda 1}\big)\big(L_1'-\sum_{\lambda=2}^{n-2}
\alpha_{\lambda}L_\lambda'\big)+\sum_{h=2}^{n-2}\big(a_{1h}+\sum_{\lambda=2}^{n-2} \alpha_{\lambda}a_{\lambda h}\big)L_h'\\
&+\big(a_{1(n-1)}+\sum_{\lambda=2}^{n-2}
\alpha_{\lambda}a_{\lambda(n-1)}\big)L_{n-1}':=\sum_{\lambda=1}^{n-1}a'_{1\lambda}L'_\lambda.
\end{split}\end{equation}
Hence
\begin{equation}\begin{split}
&a_{1(n-1)}'(z_1',0,\ov{z_1}',0)\\
=&a_{1(n-1)}(z_1',\alpha_2z_1',\cdots,\alpha_{n-2}z_1',0,
\ov{z_1'},\ov{\alpha_2z_1'},\cdots,\ov{\alpha_{n-2}z_1'},0)\\
&+\sum_{\lambda=2}^{n-2}
\alpha_{\lambda}a_{\lambda(n-1)}(z_1',\alpha_2z_1',\cdots,\alpha_{n-2}z_1',0,
\ov{z_1'},\ov{\alpha_2z_1'},\cdots,\ov{\alpha_{n-2}z_1'},0).
\end{split}\end{equation}
 Then the coefficient of ${z_1}^{'t}\ov{{z_1}^{'s}}$ with $t+s=l_0$ in $a_{1(n-1)}'(z_1',0,\ov{z_1}',0)$ is the following
\begin{equation}\begin{split}
&\sum_{\lambda=1}^{n-2}\sum_{|H|=t,|J|=s}
(a^{(l_0)}_{\lambda(n-1)})_{HJ}\alpha^{H+e_\lambda}\ov{\alpha}^J
=\sum_{\lambda=1}^{n-2}\sum_{|H|=t+1,|J|=s}
(a^{(l_0)}_{\lambda(n-1)})_{(H-e_{\lambda})J}\alpha^{H}\ov{\alpha}^J,
\end{split}\end{equation}
where $e_\lambda=(0,\cdots, 0,1,0\cdots,0)$ with  $1$  at the
$\lambda$-th position. This term is $0$ for a generic choice of
$\alpha$ if and only if $\sum_{\lambda=1}^{n-2}
(a^{(l_0)}_{\lambda(n-1)})_{(H-e_{\lambda})J}=0$. We next proceed in
two steps:

{\bf (1).} First, we suppose that there exists a pair $(H,J)$ with
$|J|\neq 0$ such that
$$\sum_{\lambda=1}^{n-2} (a^{(l_0)}_{\lambda(n-1)})_{(H-e_{\lambda})J}\not =0.$$ Then for a generic choice of $\alpha$,
 $a'^{(l_0)}_{1(n-1)}(z_1',0,\ov{z_1}',0)$ contains no non-trivial holomorphic terms.
Through the normalization procedure as in Lemma \ref{firstn}, we can
make $$a'^{(l_0)}_{1(n-1)}(z_1',\cdots,z_{n-2}',0)=0$$ and thus, in
particular,  $a'^{(l_0)}_{1(n-1)}(z_1',0,\cdots,0)=0$.  We point out
that this transformation preserves the statement in the first line
of (\ref{rhoan}). Then $a^{(l_0)}_{1(n-1)}$ and $\rho^{[m]}$ satisfy
the desired properties in (1) of Proposition \ref{propn2}. Next, we
can repeat the same argument in Lemma \ref{firstn} to normalize
$a_{j(n-1)}^{(l_0)}$ for $j\ge 2$ and thus obtain the normalization
for $a_{j(n-1)}^{(l_0)}$ with $j=2,\cdots, n-2$.

{\bf (2).} We now suppose
\begin{equation}\begin{split}\label{a100}
\sum_{\lambda=1}^{n-2}
(a^{(l_0)}_{\lambda(n-1)})_{(H-e_{\lambda})J}=0\ \text{for any}\
|H|+|J|=l_0+1,\ |J|\neq 0.
\end{split}\end{equation}
We will show that by a suitable change of coordinates of the form
$z_j'=z_j$, $z_{n-1}'=z_{n-1}+ g(z_1,\cdots,z_{n-2})$, $w'=w$, we
can make
\begin{equation}\begin{split}\label{a101}
\sum_{\lambda=1}^{n-2}
(a^{(l_0)}_{\lambda(n-1)})_{(H-e_{\lambda})0}=0\ \text{for any}\
|H|=l_0+1.
\end{split}\end{equation}
Here $g(z_1,\cdots,z_{n-2})$ is a homogeneous holomorphic polynomial
of degree $l_0+1$.

 In fact, under this transformation, we have
 \begin{equation*}\begin{split}
\frac{\p }{\p z_j}=\frac{\p }{\p z_j'}+g_{z_j}\frac{\p }{\p z_{n-1}'},\ \frac{\p }{\p z_{n-1}}=\frac{\p }{\p z_{n-1}'}.
\end{split}\end{equation*}
 Thus
  \begin{equation*}\begin{split}
  S_j=&\sum_{j=1}^{n-1}a_{jh}L_h     =\sum_{j=1}^{n-2}a_{jh}(L_h'+g_{z_h}L_{n-1}')+a_{j(n-1)}L_{n-1}'\\
     =&\sum_{j=1}^{n-2}a_{jh}L_h'+(a_{j(n-1)}+\sum_{j=1}^{n-2}a_{jh}g_{z_h})L_{n-1}'.
\end{split}\end{equation*}
Hence
\begin{equation}\begin{split}\label{ag}
a'^{(l_0)}_{\lambda(n-1)}=a^{(l_0)}_{\lambda(n-1)}+g_{z_\lambda}.
\end{split}\end{equation}
Thus $\sum_{\lambda=1}^{n-2}
(a'^{(l_0)}_{\lambda(n-1)})_{(H-e_{\lambda})0}=0$ for any $H$ with
$|H|=l_0+1$, which is equivalent to $\sum_{\lambda=1}^{n-2}z_\lambda
a'^{(l_0)}_{\lambda(n-1)}(z_1,\cdots,z_{n-2},0)=0$,  if and only if
 $$
 \sum_{\lambda=1}^{n-2}  z_\lambda g_{z_\lambda}+\sum_{\lambda=1}^{n-2}
 z_\lambda a^{(l_0)}_{\lambda(n-1)}(z_1,\cdots,z_{n-2},0)=0.
 $$
 This is the well-known Euler equation and can be solved as follows:

Notice that if we write $g=\sum_{|J|=l_0+1} \Gamma_J z^J$, then
\begin{equation*}\begin{split}
\sum_{\lambda=1}^{n-2} z_\lambda
g_{z_\lambda}=\sum_{\lambda=1}^{n-2}\sum_{|J|=l_0+1} j_\lambda
\Gamma_J z^J=(l_0+1) g.
\end{split}\end{equation*}
 Hence $g$ can be uniquely solved as
  $g=-\frac{1}{l_0+1}\sum_{\lambda=1}^{n-2} z_\lambda a^{(l_0)}_{\lambda(n-1)}(z_1,\cdots,z_{n-2},0)$. Thus we get the desired normalization property
    in (\ref{a101}). Notice that by (\ref{ag}),
    $(a^{(l_0)}_{\lambda(n-1)})_{(H-e_{\lambda})J}$ with
    $|H|+|J|=l_0+1$, $|J|\neq 0$ is not changed under this
    transformation. Hence (\ref{a100}) still holds to be true.

Notice that (\ref{a100}) and (\ref{a101})  are equivalent to the
normalization property in $(2)$ of Proposition \ref{propn2}. In
fact,
\begin{equation*}\begin{split}
&\sum_{j=1}^{n-2}z_ja^{(l_0)}_{j(n-1)}(z_1,\cdots,z_{n-2},0,\ov{z_1},\cdots,\ov{z_{n-2}},0)\\
=&\sum_{|H|+|J|=l_0+1}\sum_{j=1}^{n-2}
(a^{(l_0)}_{j(n-1)})_{(H-e_{j})J}z^{H}\ov{z}^J=0.
\end{split}\end{equation*}

This completes the proof of Proposition \ref{propn2}.
\end{proof}

We  summarize what we did in this section in the following:
\begin{cor}\label{NN} Keep the same   notations and definitions we have made so far.
After a holomorphic change of
  coordinates and after choosing  a suitable basis for $B$, we have one of the  following three normalizations
   for $\{a_{j(n-1)}^{(l_0)},\rho^{[m]}\}_{j=1}^{n-2}$ under the assumption that  $l_0\leq a^{(n-2)}(M,0)-1$:
\begin{enumerate}
\item[(I)]   $a_{j(n-1)}^{(l_0)}(z_1,\cdots,z_{n-2},0,\ov{z_1},\cdots,\ov{z_{n-2}},0)$
  is holomorphic in $z_1,\cdots,z_{n-2}$ \hbox{ for each}\ j, and
  there exits
  $ j_0\in [2,n-2]$\ such that $a^{(l_0)}_{j(n-1)}(z_1,\cdots,z_{n-2},0,\ov{z_1},\cdots,\ov{z_{n-2}},0)=0$
  \ \hbox{for}\  $1\leq j\leq j_0-1,$
  $a^{(l_0)}_{j_0(n-1)}(0,\cdots,0,z_{j_0},\cdots,z_{n-2},0,\cdots,0)=0$,
  but
      \\ $a^{(l_0)}_{j_0(n-1)}(z_1,\cdots,z_{n-2},0,\cdots,0)\not\equiv
      0$.

\item[(II)]  $a^{(l_0)}_{1(n-1)}(z_1,0,\cdots,0,\ov{z_1},0,\cdots,0)\not \equiv 0$,\
  $a^{(l_0)}_{1(n-1)}(z_1,0,\cdots,0)=0$,\\
  $\rho^{[m]}(z_1,0,\cdots,0,z_{n-1},0,\ov{z_1},0,\cdots,0,\ov{z_{n-1}},0)$ is not identically zero (
  and contains no non-trivial holomorphic terms).

  \item[(III)] For a certain $j\in [1,n-2]$, $a^{(l_0)}_{j(n-1)}(z_1,\cdots,z_{n-2},0,\ov{z_1},\cdots,
  \ov{z_{n-2}},0) $ is not holomorphic,
  $\sum_{j=1}^{n-2}z_ja^{(l_0)}_{j(n-1)}(z_1,\cdots,z_{n-2},0,\ov{z_1},\cdots,\ov{z_{n-2}},0)=0$,
  and
   $$\rho^{[m]}(z_1,\cdots,z_{n-1},0,\ov{z_1},\cdots,\ov{z_{n-1}},0)$$ is not
   identically  zero  containing no non-trivial holomorphic  terms.

\end{enumerate}

\end{cor}

\section{Proof of Theorem \ref{mainthm}}

 In this section, we  present  a proof of Theorem \ref{mainthm},
 assuming    Lemma \ref{lem4-tech} and Theorem \ref{main-tech} whose
 proofs are long and will be given in $\S 5$ and $\S 6$.

\begin{proof}[Proof of the equality: $t^{(n-2)}(M,p)=a^{(n-2)}(M,p)$]
We keep the notations set up in $\S 2$ and $\S 3$.  Assume that $M$
is defined as in (\ref{rho}) and (\ref{R}). As we mentioned there,
we assume that $ a^{(n-2)}(M,p=0)<\infty$. Supposing that
$t^{(n-2)}(M,0)> a^{(n-2)}(M,0)$, we will then seek  a
contradiction.

Let $B$ be an $(n-2)$-dimensional smooth vector subbundle of
$T^{1,0}M$ such that $t^{(n-2)}(M,0)=t^{(n-2)}(B,0)$. By the
assumption that $t^{(n-2)}(M,0)>a^{(n-2)}(M,0)$,  for  any $l\leq
a^{(n-2)}(M,0)$ we have
\begin{equation}\label{lieb}
\langle F,\p \rho \rangle(0)=0\ \text{for any }\
F=[F_{l},F_{l-1},\cdots[F_2,F_1]\cdots]\ \text{with}\
F_1,\cdots,F_{l}\in \mathcal{M}_1(B).
\end{equation}
We also assume that  the local sections of $B$ are generated by
$S_1,\cdots,S_{n-2}$.

 Recall that the weight of $z_j$ for $1\leq j\leq
n-2$ and their conjugates is  $1$. Define  the weight of $z_{n-1}$
and its conjugate  to be $k=l_0+1$. Denote the weight of $w$ to be
$m$, which is the lowest weighted vanishing order of
$\rho(z,\ov{z},0)$ with respect to the weights just given. We also
define
\begin{equation}
\begin{split}
&\text{wt}(\frac{\p }{\p z_j})=\text{wt}(\frac{\p }{\p \ov{z_j}})=-1 \ \text{for}\ 1\leq j\leq n-2,\\
& \text{wt}(\frac{\p }{\p z_{n-1}})=\text{wt}(\frac{\p }{\p \ov{z_{n-1}}})=-k,\ \text{wt}(\frac{\p }{\p w})=\text{wt}(\frac{\p }{\p \ov{w}})=-m.
\end{split}
\end{equation}

By the definition of $a^{(n-2)}(M,0)$, when restricted to the
$(n-2)$-manifold $\{(z,w):\ z_{n-1}=w=0\}$, the vanishing order of
$\rho$ is bounded by $a^{(n-2)}(M,0)$. Thus $m\leq a^{(n-2)}(M,0)$.
When $k\leq a^{(n-2)}(M,0)$, we assume  that $S_j$ and $\rho$ are
normalized as in Corollary \ref{NN}.

Write
$$
S_j^0=\frac{\p}{\p z_j}+a_{j(n-1)}^{(k-1)} \frac{\p}{\p
z_{n-1}}+a_{jn}^{[m-1]} \frac{\p}{\p w}.
$$
Then $S_j^0$ is the sum of  terms in $S_j$ of weighted degree $-1$.

Now, let $\mathcal{M}^0$ be the $C^\infty(M^0)$-module spanned by
$S_j^0$ and $\ov{S_j^0}$ for all $1\leq j \leq n-2$, where
$M^0=\{(z,w):\ \rho^{[m]}=-2\text{Re}w+\chi^{[m]}(z,\ov{z},0)=0\}$
and $\mathcal{M}^0_l$ be the $C^\infty(M^0)$ module formed by taking
the Lie bracket of length $\leq l$ of sections from $\mathcal{M}^0$
for $l=2,\cdots$. $\mathcal{M}^0_\infty=\cup_{l\in
\mathbb{N}}\mathcal{M}^0_l$. Now we need the following two lemmas.

\begin{lem}\label{klm}
  It holds that $k<m$.
\end{lem}

\begin{proof}
  Suppose that $k\geq m$. Then the weight of $z_{n-1}$ is no less than $m$. Hence $\chi^{[m]}(z,\ov{z},0)$ is independent of $z_{n-1}$.
  Write
  $$
  \w{S_j^0}=S_j^0-a_{j(n-1)}^{(k-1)}L_{n-1}^0:=\frac{\p}{\p z_{j}}+\w{a_{jn}}\frac{\p}{\p {w}}.
  $$
  Since $S_j^0$ is tangent to $M^0$, whose defining function is independent of $z_{n-1}$, we see that
   $\w{S_j^0}$ is tangent to $M^0$ and $\w{a_{jn}}=\frac{\p \chi^{[m]}}{\p z_j}$. Hence
     $\w{a_{jn}}$ is independent of $z_{n-1}$.

     Regarding $M^0$ as  a real hypersurface in ${\mathbb C}^{n-1}$.
  Let $\w{\mathcal{M}}^0$ be the $C^\infty(M^0)$ module spanned by $\w{S_j^0}$ and $\ov{\w{S_j^0}}$
   for all $1\leq j \leq n-2$.
  Define
  $Q: {\mathcal{M}}^0\rightarrow \w{\mathcal{M}}^0$ by sending $\sum_{j=1}^{n-1} d_jL_j^0\in {\mathcal{M}}^0$ to
  $\sum_{j=1}^{n-2} d_jL_j^0\in \w{\mathcal{M}}^0$. Then by (\ref{lieb}), for any $Z_j^0\in \w{\mathcal{M}}^0$, there
   exists $Y_j^0\in {\mathcal{M}}^0$ with $Q(Y_j^0)=Z_j^0$ such that
  $$
  \langle [Z_j^0,[Z^0_{j_1},\cdots,[Z^0_2,Z^0_1]\cdots],\p \rho\rangle(0)=\langle [Y_j^0,[Y^0_{j-1},\cdots,[Y^0_2,Y^0_1]\cdots],\p \rho\rangle(0)=0\ \text{for}\ j\leq m.
  $$ (Indeed, we can simply take $Y_j^0$ to be $Z_j^0$,
   but regard it as a CR vector field of $M^0$ as a real hypersurface in ${\mathbb C}^n$.) Hence we have
  $t^{((n-1)-1)}(M^0,0)>m$. However, by our construction, $a^{((n-1)-1)}(M^0,0)=m$.
   This contradicts a result of Bloom-Graham for the $(n-1)$ types
   in \cite{BG1}, which says that $t^{((n-1)-1)}(M^0,0)=a^{((n-1)-1)}(M^0,0)$ for $M^0\subset {\mathbb C}^{n-1}$.
\end{proof}

\begin{lem}\label{rhom0}
  For any $ Y^0\in \mathcal{M}^0_l$, we have $\langle Y^0,\p \rho^{[m]} \rangle(0)=0$.
\end{lem}

 \begin{proof} Assume that $Y^0=[X_l^0,\cdots,[X_2^0,X_1^0]\cdots]$ with $X_j^0\in \mathcal{M}^0$. Write
$$
X_j^0=Z_j^0+B_j\frac{\p}{\p w}+C_j\frac{\p}{\p \ov{w}}\ \text{with}\
Z_j^0=\sum_{k=1}^{n-1}(b_{jk}\frac{\p}{\p z_k}+c_{jk}\frac{\p}{\p
\ov{z_k}}).
$$
Here $Z_j^0$ is weighted homogeneous of degree $-1$ and
wt$(B_j)=\text{wt}(C_j)=m-1$. A direct computation shows
$$
[X_2^0,X_1^0]=(Z_2^0(B_1)-Z_1^0(B_2))\frac{\p }{\p w}\ \text{mod}\ (\frac{\p}{\p z},\frac{\p}{\p \ov{z}},\frac{\p}{\p \ov{w}})
$$
and by an induction,
$$
Y^0=C_l^0\frac{\p }{\p w}\ \text{mod}\ (\frac{\p}{\p z},\frac{\p}{\p
\ov{z}},\frac{\p}{\p \ov{w}})
$$
with  $C_l^0$ a weighted homogeneous polynomial of weighted degree
equal to $-l+m$. Hence $Y^0\equiv 0$ when $l>m$ and $Y^0|_0=0$ when
$l<m$ mod $(\frac{\p}{\p z},\frac{\p}{\p \ov{z}},\frac{\p}{\p
\ov{w}})$.

When $l=m$, Suppose that $Z_j\in \mathcal{M}_1$ such that
$(Z_j)^0=X_j^0$. Then $[Z_j,Z_k]^{0}=[X_j^0,X_k^0]$. Hence if $Z\in
\mathcal{M}_l$ with $l=m$ such that $(Z)^0=Y^0$, then
 $Z=C_l^0Y^0+D_l\frac{\p}{\p w}$  mod $(\frac{\p}{\p z},\frac{\p}{\p \ov{z}},\frac{\p}{\p \ov{w}})$ with
$\text{wt}(D_l)>\text{wt}(C_l^0)$. From (\ref{lieb}), $Z|_0=0$ mod
$(\frac{\p}{\p z},\frac{\p}{\p \ov{z}},\frac{\p}{\p \ov{w}})$. Thus
we obtain $C_m^0\equiv 0$. Hence $\langle Y^0,\p \rho^{[m]}
\rangle(0)=0$ for all $l\in \mathbb{N}$.
\end{proof}
Then we  have $\frac{\p }{\p v}|_0\not\in \mathcal{M}^0_\infty$. By
the Nagano theorem (see [BER], for instance), $\mathcal{M}^0_\infty$
gives a unique real analytic integral submanifold $N^0$ with $0\in
N^0\subset M^0=\{-2\text{Re}w+\chi^{[m]}(z,\ov{z},0)=0\}$. Moreover,
dim$_{\mathbb{R}}N^0=\text{dim}_{\mathbb{R}}\text{Re}\mathcal{M}^0_\infty|_{N^0}$.
Since $\frac{\p}{\p v}|_0\not\in T_0N^0$, $N^0$ is contained in the
graph of $v=f_1(z,\ov{z},u)$ for a certain real analytic function
$f_1$ near $0$. Since $u=\frac{1}{2}\chi^{[m]}(z,\ov{z},0)$, we
conclude that $N^0$ is contained in the graph of
$$
w=f(z,\ov{z})=\frac{1}{2}\chi^{[m]}(z,\ov{z},0)+if_1\big(z,\ov{z},\frac{1}{2}\chi^{[m]}(z,\ov{z},0)\big).
$$

We mention that from the pseudoconvexity of $M$, we immediately
conclude the pseudoconvexity of $M^0$, which is equivalent to the
plurisubharmonicity of $\hbox{Re}(f)=\chi^{[m]}(z,\ov{z},0)$.

\begin{lem}
  The real dimension of $N^0$ is either $2n-3$ or $2n-2$.
\end{lem}

\begin{proof}
  The proof is carried out in two steps according to the properties of $a_{j(n-1)}^{(k-1)}(z,\ov{z})$ in  Proposition \ref{propn}.
  \medskip

  {\bf (1)}: Suppose we have the normalization in (I) of Proposition \ref{propn}. We suppose that  $(a_{j_0(n-1)}^{(k-1)})_{H+e_{\mu}}\neq 0$
   with  $H=(h_1,\cdots,h_{n-2})$ and $1\leq \mu\leq j_0-1$.
Then
$$
[S_\mu^0,S_{j_0}^0]=\frac{\p}{\p
z_{\mu}}(a_{j_0(n-1)}^{(k-1)})\frac{\p}{\p z_{{n-1}}}\ \text{mod} \
(\frac{\p }{\p w},\frac{\p }{\p \ov{w}}).
$$
Write
\begin{equation}\begin{split}\label{sh}
&(\stackrel{h_1\
\text{times}}{\overbrace{S_1^0,\cdots,S_1^0}},\stackrel{h_2 \
\text{times}}{\overbrace{S_2^0,\cdots,S_2^0}},\cdots,\stackrel{h_{n-2}\
\text{times}} {\overbrace{S_{n-2}^{0},\cdots,S_{n-2}^{0}}})\
\text{as}\ (X_1,\cdots, X_{|H|}).
\end{split}\end{equation}
Then
\begin{equation*}\begin{split}
&[X_{1},[\cdots [X_{|H|}, [S_\mu^0,S_{j_0}^0]]\cdots]]\\
=&(h_\mu+1)\cdot h_1!\cdots
h_{n-2}!(a_{j_0(n-1)}^{(k-1)})_{H+e_{\mu}} \frac{\p}{\p z_{{n-1}}}\
\text{mod} \ (\frac{\p }{\p w},\frac{\p }{\p \ov{w}}) .
\end{split}\end{equation*}
Since its conjugate is also in $\mathcal{M}_\infty^0$, we conclude
that the dimension of $N^0$ is $2n-2$.
\bigskip

{\bf (2)}: Suppose we have the normalization in (II) of Proposition
\ref{propn}. Then there is a
$(H,J)=(h_1,\cdots,h_{n-2},j_1,\cdots,j_{n-2} ) $ such that
$(a_{1(n-1)}^{(k-1)})_{H(J+e_\mu)}\neq 0$. Then
$$
[\ov{S_\mu^0},S_1^0]=\frac{\p}{\p \ov{z_\mu}
}(a_{1(n-1)}^{(k-1)})\frac{\p}{\p {z_{{n-1}}}}\ \text{mod} \
(\frac{\p}{\p \ov{z_{{n-1}}}},\frac{\p }{\p w},\frac{\p }{\p
\ov{w}}).
$$

Write $(X_1,\cdots, X_{|H|})$ as in (\ref{sh}) and write
\begin{equation*}\begin{split}
&(\stackrel{j_1\
\text{times}}{\overbrace{S_1^0,\cdots,S_1^0}},\stackrel{j_2\
\text{times}}{\overbrace{S_2^0,\cdots,S_2^0}},\cdots,\stackrel{j_{n-2}\
\text{times}} {\overbrace{S_{n-2}^{0},\cdots,S_{n-2}^{0}}})\
\text{as}\ (Y_1,\cdots, Y_{|J|}).
\end{split}\end{equation*}
Then

\begin{equation*}\begin{split}
Y_{HJ}:=&[X_1,[\cdots, [X_{|H|}, [\ov{Y_1},[\cdots,[\ov{Y_{|J|}},[\ov{S_\mu^0},S_1^0]]\cdots]\\
=&(j_{\mu}+1)\cdot h_1!\cdots h_{n-2}!j_1!\cdots
j_{n-2}!(a_{j(n-1)}^{(k-1)})_{H(J+e_{\mu})} \frac{\p}{\p
{z_{{n-1}}}} \text{mod} \ (\frac{\p}{\p \ov{z_{{n-1}}}},\frac{\p
}{\p w},\frac{\p }{\p \ov{w}}).
\end{split}\end{equation*}
Hence $Y_{HJ} \not\in \text{span}_{\mathbb{C}}\{S_j^0,\ov{S_j^0},
1\leq j\leq n-2\}$. Thus either Re$Y_{HJ}|_0\neq 0$ or
Im$Y_{HJ}|_0\neq 0$.

 Since $\frac{\p}{\p w}|_0$ and $\frac{\p}{\p \ov{w}}|_0$ are not tangent to $N^0$ at $0$,
  the dimension of $N^0$ is either $2n-3$ or $2n-2$.
\end{proof}

\begin{lem} When $N^0$ has real dimension $2n-2$, $f$ is a
weighted homogeneous polynomial of weighted degree $m$.
\end{lem}
\begin{proof} Let $X^0$ be a (weighted) homogeneous vector field from
$\mathcal{M}^0_\infty$. Then from the equality that
$X^0(-w+f)=\ov{X^{0}}({-w+f})\equiv 0$, it follows that
$X^0(-w+f^{[m]})=X^{0}(\ov{-w+f^{[m]}})\equiv 0$. Hence the manifold
defined by $w=f^{[m]}$ is also an integral manifold of the module
$\mathcal{M}^0_\infty$ through $0$. By the uniqueness of the
integrable manifold, we conclude that $f^{[m]}=f$.
\end{proof}

 The rest of the argument is carried out according to the  dimension of $N^0$. We remark that when
  the real
  dimension of $N^0$ is
 $2n-3$, it is a CR submanifold of hypersurface type, for it has a constant CR dimension $n-2$ everywhere.
 When its dimension is
 $2n-2$, it has CR dimension $n-1$ at the origin. Since it cannot be Levi-flat due to the fact that
  $\hbox{Re}(f)\not \equiv 0$, it is thus  a codimension two CR singular submanifold.

\medskip
{\bf Step I.} In this step, we suppose $N^0$ is of real dimension $2n-2$.
Since $\ov{S^0_j}$ is tangent to $N^0$,
and since $N^0$ is defined by $w=f(z,\ov{z})$ for $z\approx 0$ in
${\mathbb C}^{n-1}$, we have
\begin{equation}\label{sf}
\frac{\p }{\p
\ov{z_j}}f(z,\ov{z})+\ov{a_{j(n-1)}^{(k-1)}(z_1,\cdots,z_{n-2},\ov{z_1},\cdots,\ov{z_{n-2}})}
\frac{\p }{\p \ov{z_{n-1}}}f(z,\ov{z})=0,\ z \in\mathbb{C}^{n-1}.
\end{equation}

 By Lemma \ref{klm}, we  have $k<m\leq a^{(n-2)}(M,0)$. Our next discussions are divided into the
 following cases according to the normalization in Corollary \ref{NN}.\medskip

{\bf Case (1):} In this case, suppose that we have the normalization
in (1) of Corollary \ref{NN}. For $1\leq j\leq j_0-1$,
$a_{j(n-1)}^{(k-1)}\equiv 0$. Thus (\ref{sf}) takes the form
$\frac{\p f}{\p \ov{z_j}}=0$. Hence $f$ is holomorphic in
$z_1,\cdots,z_{j_0-1}$. By Lemma \ref{lem4-tech} to be proved in $\S
5$,  since $\hbox{Re}(f)$ is plurisubharmonic and contains
non-trivial holmorphic terms, $f$ is in fact independent of
$z_1,\cdots,z_{j_0-1}$. Setting $j=j_0$ in (\ref{sf}), we obtain
$$
\frac{\p f}{\p \ov{z_{j_0}}}=-\ov{a^{(k-1)}_{j_0(n-1)}}\frac{\p f}{\p \ov{z_{n-1}}}.
$$
 Notice that the left hand side is independent of $z_1,\cdots,z_{j_0-1}$. On the other hand,
  the right hand side is  divisible by $\ov{a^{(k-1)}_{j_0(n-1)}}$, in which each term depends on $z_1,\cdots,z_{j_0-1}$
   or their conjugates. Thus $\frac{\p f}{\p \ov{z_{j_0}}}=\frac{\p f}{\p \ov{z_{n-1}}}=0$.
   Substituting this back to (\ref{sf}), we obtain $\frac{\p f}{\p \ov{z_{j}}}=0$ for each $1\leq j\leq n-2$.
   Thus $f$ is holomorphic in $z_1,\cdots,z_{n-1}$.
       However, $\chi^{[m]}=\text{Re}(f)\neq 0$ does not contain any non-trivial holomorphic term.
          We thus reach a contraction.
\medskip

{\bf Case (2):} In this case,  suppose we have the normalization in
(2) of  Corollary \ref{NN}. Letting $j=1$ in (\ref{sf}) and
restricting the equation to $z_1$ and $z_{n-1}$ spaces, we obtain:
 \begin{equation}\label{sf*}
\big(\frac{\p f}{\p \ov{z_1}}+\ov{a_{1(n-1)}^{(k-1)}}\frac{\p }{\p
\ov{z_{n-1}}}f\big)(z_1,0,\cdots,0,z_{n-1},\ov{z_1},0,\cdots,0,\ov{z_{n-1}})=0.
\end{equation}
By our assumption,
$a_{1(n-1)}^{(k-1)}(z_1,0,\cdots,0,z_{n-2},\ov{z_1},0,\cdots,0,\ov{z_{n-2}})$
is not identically zero and contains no non-trivial holomorphic
terms. By Theorem \ref{main-tech}, we know $\chi^{[m]}=\text{Re}(f)=
0$ when restricted to $z_1$ and $z_{n-1}$ spaces. This contradicts
the last normalization in (2) of Corollary  \ref{NN}.
\medskip

{\bf Case (3):} In this case, suppose we have the normalization in
(3) of Corollary  \ref{NN}. Then we have
$\sum_{j=1}^{n-2}z_ja^{(k-1)}_{j(n-1)}(z_1,\cdots,z_{n-2},\ov{z_1},\cdots,\ov{z_{n-2}})=0$.
Since $f_{\ov{z_j}}+\ov{a^{(k-1)}_{j(n-1)}}f_{\ov{z_{n-1}}}=0$ and
$a^{(k-1)}_{j(n-1)}$ is independent of $z_{n-1}$ and $w$, we get
$$
\sum_{j=1}^{n-2} \ov{z_j}f_{\ov{z_j}}(z_1,\cdots,z_{n-1},\ov{z_1},\cdots,\ov{z_{n-1}})=0.
$$
This is again the well-known Euler equation on $f$.  Write
$f(z,\ov{z})=\sum_{|\alpha|\geq 0}g_\alpha(z)\ov{z}^{\alpha}$, where
$g(z)$ is holomorphic in $z$. Then
$$
\sum_{j=1}^{n-2}
\ov{z_j}f_{\ov{z_j}}=\sum_{j=1}^{n-2}\sum_{|\alpha|\geq
0}g_\alpha(z)\alpha_j\ov{z}^{\alpha}=\sum_{|\alpha|\geq
0}(\sum_{j=1}^{n-2}\alpha_j)g_\alpha(z)\ov{z}^{\alpha}=0.
$$
Hence $g_\alpha(z)=0$ for $\sum_{j=1}^{n-2}|\alpha_j|>0$. Thus
$f(z_1,\cdots,z_{n-1},\ov{z_1},\cdots,\ov{z_{n-1}})$ is holomorphic
in $z_1,\cdots,z_{n-2}$. Hence $f_{\ov{z_j}}=0$ for each $1\leq
j\leq n-2$. Substituting this back to
$f_{\ov{z_j}}+\ov{a^{(k-1)}_{j(n-1)}}f_{\ov{z_{n-1}}}=0$, we know
$\ov{a^{(k-1)}_{j(n-1)}}f_{\ov{z_{n-1}}}=0$. Recall that at least
one $a^{(k-1)}_{j(n-1)}$ is not holomorphic and thus is nonzero.
Thus $f_{\ov{z_{n-1}}}=0$. Hence
$f(z_1,\cdots,z_{n-1},\ov{z_1},\cdots,\ov{z_{n-1}})$ is holomorphic
in $z_1,\cdots,z_{n-1}$.  Since Re$f$ contains no non-trivial
holomorphic terms, we reach a contradiction.

\medskip
{\bf Step II.} In this step, we suppose $N$ is of real dimension
$2n-3$. \medskip

Without loss of generality, we  assume Re$Y_{HJ}|_0\neq 0$. Then
$$\mathbb{C}TN^0=\text{Span}_{\mathbb{C}}\{S_1^0,\cdots,S_{n-2}^0,\ov{S_1^0},\cdots,\ov{S_{n-2}^0},\text{Re}
Y_{HJ}\} \ \text{near}\ 0.$$
 Thus $N^0$ is a CR manifold of
hypersurface type of finite type in the sense of
H\"{o}mander-Bloom-Graham. With a rotation in $z_{n-1}$-variable, we
can assume that Re$Y_{HJ}|_0=\frac{\p}{\p x_{n-1}}|_0$. Now, we
define $\pi: N^0\rightarrow \mathbb{C}^{n-1}$ by sending
$(z_1,\cdots,z_{n-1},w)$ to $(z_1,\cdots,z_{n-1})$. $\pi$ is a CR
immersion near $0$. Write $\pi(N^0)=\w{N^0}\subset
\mathbb{C}^{n-1}$. Then $\w{N^0}$ is a real hypersurface in
$\mathbb{C}^{n-1}$ and $\pi^{-1}:\w{N^0}\rightarrow N^0$ is a local
real analytic CR diffeomorphism with $\pi^{-1}(0)=0$. Write
$$
\pi^{-1}(z_1,\cdots,z_{n-1})=(z_1,\cdots,z_{n-1},h(z_1,\cdots,z_{n-1})).
$$
Since real analytic CR functions are restrictions of  holomorphic
functions, we can assume that $h(z_1,\cdots,z_{n-1})$ is a
holomorphic function. Notice that $h=O(|z|^2)$ and define
$(\xi_1,\cdots,\xi_{n-1},\eta)=F(z_1,\cdots,z_{n-1},w)=(z_1,\cdots,z_{n-1},w-h(z_1,\cdots,z_{n-1}))$.
Then
$$
F(N^0)\subset \mathbb{C}^{n-1}\times \{0\}=\{(\xi_1,\cdots,\xi_{n-1},0):\ \xi_1,\cdots,\xi_{n-1}\in \mathbb{C}\}.
$$

Also, $F(M^0)$ is defined by
$-2\text{Re}\eta+2\text{Re}h(\xi)+\chi^{[m]}(\xi,\ov{\xi},0)=0$ or
$2\text{Re}\eta=2\text{Re}h(\xi)+\chi^{[m]}(\xi,\ov{\xi},0)=\w{\rho}(\xi,\ov{\xi})$.
Notice that $F(M^0)$ is holomorphically equivalent to $M^0$. Hence
$F(M^0)$ is also pseudo-convex and of finite type in the sense of
H\"{o}mander-Bloom-Graham. Notice that $\w{N^0}=F(N^0)\subset
\w{M_0}=F(M^0)$. Hence, $\forall \xi \in \w{N^0}$,
$\w{\rho}(\xi,\ov{\xi})=0$. Notice that $\w{\rho}=O(|\xi|^2)$ and is
plurisubharmonic. By the following proposition, we reach a
contradiction to the assumption that
$2\text{Re}h(\xi)+\chi^{[m]}\not\equiv 0$.

\begin{prop}\label{disc}
 Let $N$ be a real analytic  hypersurface  in $\mathbb{C}^{n-1}$ with $0\in N$ with $n\ge 3$.
  Let $\rho(z,\ov{z})$ be a real analytic plurisubharmonic function with $\rho=O(|z|^2)$ as $z\rightarrow 0$
  defined over a neighborhood of ${\mathbb C}^{n-1}$. Assume that $N$ is of finite type in the sense of
  H\"{o}mander--Bloom-Graham and $N\subset \{\rho=0\}$. Then $\rho\equiv 0$.
\end{prop}

\begin{proof}
Let $\phi:\Delta\rightarrow \mathbb{C}^{n-1}$ be a smooth small
holomorphic disk attached to $N$ with $\phi(1)=0$. Namely, we assume
that $\phi\in C^\infty(\ov{\Delta})\cap \text{Hol}(\Delta)$,
$\phi(\p \Delta)\subset N$, $\phi(1)=0$, $\phi(\ov{\Delta})$ is
close to $0$. Since $\rho(\phi(\xi),\ov{\phi(\xi)})=0$ on $\p
\Delta$ and $\frac{\p}{\p \xi\p
\ov{\xi}}\rho(\phi(\xi),\ov{\phi(\xi)})\geq 0$ for $\xi\in \Delta$,
$\rho(\phi(\xi),\ov{\phi(\xi)})$ is a subharmonic function in
$\Delta$ smooth up to $\p \Delta$. By the maximum principle, we have
$\rho(\phi(\xi),\ov{\phi(\xi)})<0$ for $\xi\in \Delta$ unless
$\rho(\phi(\xi),\ov{\phi(\xi)})\equiv0$ for $\xi\in \Delta$. Now, we
apply the Hopf Lemma to get
$$
\frac{d}{d \xi}\rho(\phi(\xi),\ov{\phi(\xi)})|_{\xi=1}\ge 0
$$
 and the equality holds if and only if $\rho(\phi(\xi),\ov{\phi(\xi)})\equiv0$. On the other hand, $$\rho(\phi(\xi),\ov{\phi(\xi)})=O(|\phi(\xi)|^2)=O(|\phi(\xi)-\phi(1)|^2)=O(|\xi-1|^2)$$
as $\xi(\in (0,1))\rightarrow 1$. We conclude that $\rho(\phi(\xi),\ov{\phi(\xi)})\equiv0$.

Next, by a result of Tr\'epreau \cite{Tr}, since the union
$\phi(\Delta)$ of all attached discs fill in at least one side of
$N$ near $0$, we see that $\rho\equiv 0$ in one side of $N$. Since
we assumed that $\rho$ is real analytic, we conclude that
$\rho\equiv0$. This completes the proof of Proposition \ref{disc}.
\end{proof}

We thus complete the proof of the equality that
$t^{(n-2)}(M,p)=a^{(n-2)}(M,p)$.
\end{proof}


\begin{proof}[Proof of the equality: $c^{(n-2)}(M,p)=a^{(n-2)}(M,p)$]

We continue to use the notations and initial setups as in $\S 2$ and
$\S 3$. By \cite{Bl2}, we  have $c^{(n-2)}(M,p=0)\geq
a^{(n-2)}(M,p=0)$. We will seek a contradiction supposing that
$c^{(n-2)}(M,0)> a^{(n-2)}(M,0)$.

Let $B$ be an $(n-2)$-dimensional smooth subbundle of $T^{1,0}M$
such that $c^{(n-2)}(M,0)=c^{(n-2)}(B,0)$.
Repeating the normalization procedures as in $\S 3$, we can find a
basis $\{S_j\}$ of $B$
  and  a defining function $\rho$
 that satisfy  the   normalization conditions  as in Corollary \ref{NN}.
  Since $c^{(n-2)}(M,0)>a^{(n-2)}(M,0)$,  for any $l\leq
  a^{(n-2)}(M,0)$, we have
\begin{equation}\label{levid}
F_1\cdots F_{l-2}\sum_{j=1}^{n-2}\p\ov{\p}\rho(S_j,\ov{S_j})=0\
\text{for any}\ F_1,\cdots,F_{l-2}\in \mathcal{M}_1(B).
\end{equation}

As in the proof of $t^{(n-2)}(M,p)=a^{(n-2)}(M,p)$, we can similarly
define the weights of $z_1,\cdots,z_{n-1},w$, and define  $S_j^0$,
$\mathcal{M}^0$, $M^0$, $\mathcal{M}^0_l$, $\mathcal{M}^0_\infty$.
By  the same argument as that in  Lemma \ref{klm}, we have $k<m$.
Similar to Lemma \ref{rhom0}, we have the following:

\begin{lem}\label{hlevi}
  For any $l$ and  $ Y^0_1,\cdots,Y^0_{l-2}\in \mathcal{M}^0_1$, we have
  $$Y^0_1\cdots
  Y^0_{l-2}\sum_{j=1}^{n-2}\p\ov{\p}\rho^{[m]}(S^0_j,\ov{S^0_j})(0)=0.$$
\end{lem}
\begin{proof}
First notice that $Y^0:=Y^0_1\cdots
Y^0_{l-2}\sum_{j=1}^{n-2}\p\ov{\p}\rho^{[m]}(S^0_j,\ov{S^0_j})$ is a
weighted homogeneous polynomial of weighted degree $-l+m$. Hence
$Y^0=0$ when $l>m$ and $Y^0|_0=0$ when $l<m$.

  Next we suppose $l=m$. For any $1\leq j\leq l-2$, suppose $Z_j\in \mathcal{M}_1$ such that $(Z_j)^0=Y_j^0$. By (\ref{levid}),
   we have
$$
Z_1\cdots Z_{m-2}\sum_{j=1}^{n-2}\p\ov{\p}\rho(S_j,\ov{S_j})(0)=0.
$$
Notice that
$$
Z_1\cdots
Z_{m-2}\sum_{j=1}^{n-2}\p\ov{\p}\rho(S_j,\ov{S_j})=Y^0_1\cdots
Y^0_{m-2}\sum_{j=1}^{n-2}\p\ov{\p}\rho^{[m]}(S^0_j,\ov{S^0_j})+o(1).
$$
We thus have $Y^0(0)=0$ for $l=m$. This completes the proof of Lemma
\ref{hlevi}.
\end{proof}
Now we similarly apply  the Nagano theorem to conclude that
$\mathcal{M}^0_\infty$ gives a unique real analytic integral
submanifold $N^0$ with $N^0\subset
M^0=\{-2\text{Re}w+\chi^{[m]}(z,\ov{z},0)=0\}$. Since the tangent
space
 at each point of $N^0 $ is generated by Re$\mathcal{M}^0_\infty$,
by Lemma \ref{hlevi}, we have
$$
\sum_{j=1}^{n-2}\p\ov{\p}\rho^{[m]}(S^0_j,\ov{S^0_j})\equiv 0\
\text{on}\ N^0,
$$
for $\rho^{[m]}(S^0_j,\ov{S^0_j})$ is real-analytic  and it vanishes
to infinite order at $0$ along $N^0$.
Since $\rho^{[m]}$ is plurisubharmonic, we
have $\p\ov{\p}\rho^{[m]}(S^0_j,\ov{S^0_j})\geq 0$ on $M^0$. Notice
that $N^0\subset M^0$, we have
$\p\ov{\p}\rho^{[m]}(S^0_j,\ov{S^0_j})\equiv 0$ on $N^0$. Hence
$\hbox{Re}(S_j^0), \hbox{Im}(S_j^0)\in T^N(N^0)$. By
\cite[Proposition 2]{DF}, for any vector field $Y^0$ in
$\mathcal{M}^0_j$, $\hbox{Re}(Y^0), \ \hbox{Im}(Y^0)\in T^N(N^0)$
for each $j$. Hence for any $ Y^0\in \mathcal{M}^0_j$, we have
$\langle Y^0,\p \rho^{[m]} \rangle(0)=0$, for both the real part and
the imaginary part of $Y^0|_0$ are in $\hbox{Re}(T_0^{(1,0)}N^0)$.
This then reduces the rest of the proof to that in the proof of the
equality of $t^{(n-2)}(M,0)=a^{(n-2)}(M,0)$. The proof of the
equality $c^{(n-2)}(M,p)=a^{(n-2)}(M,p)$ is now complete.
 \end{proof}

\section{Applications  of positivity: Proofs of four lemmas}

In this section, we prove four lemmas concerning  a homogeneous
polynomial whose real part is plurisubharmonic.  These lemmas  will
be used for the proof of Theorem \ref{main-tech} in $\S 6$. (Lemma
\ref{lem4-tech} was also used in $\S 3$).
We begin with the following:

\begin{lem}\label{lem1-teh}
  Let $h(\xi,\ov{\xi})$ be a homogeneous polynomial of $(\xi,\ov{\xi})\in \mathbb{C}\times \mathbb{C}$. Suppose that
\begin{equation}\begin{split}\label{heq}
hh_{\xi\ov{\xi}}-h_{\xi}h_{\ov{\xi}}=0.
\end{split}\end{equation}
Then $h$   must be a monomial. Namely, $h=c\xi^j\ov{\xi}^k$ for a
certain complex number $c$.

\end{lem}

\begin{proof} This lemma may be known to experts. We give a simple
proof here for convenience of a reader.
  Suppose that $h$ is not a monomial and takes the following form:
$$
h=\alpha \xi^j\ov{\xi}^h+\beta \xi^t\ov{\xi}^s+O(\xi^{t+1})\
\text{with}\ j<t,\ \alpha,\beta\neq 0.
$$
Here and in what follows, we write $O(\xi^k)$ for a homogeneous
polynomial with degree in $\xi$ at least $k$. Then
\begin{equation*}
\left(
\begin{array}{ll}
h & h_{\xi}\\
h_{\ov{\xi}}& h_{\xi\ov{\xi}}
\end{array}\right)=\left(
\begin{array}{ll}
\alpha \xi^j{\ov{\xi}}^h+\beta \xi^t{\ov{\xi}}^s+O(\xi^{t+1}) & j\alpha \xi^{j-1}{\ov{\xi}}^h+t\beta \xi^{t-1}{\ov{\xi}}^s+O(\xi^{t})\\
h\alpha \xi^j{\ov{\xi}}^{h-1}+s\beta
\xi^t{\ov{\xi}}^{s-1}+O(\xi^{t+1})& jh\alpha
\xi^{j-1}{\ov{\xi}}^{h-1}+ts\beta
\xi^{t-1}{\ov{\xi}}^{s-1}+O(\xi^{t})
\end{array}\right).
\end{equation*}
Thus
\begin{equation*}\begin{split}
hh_{\xi\ov{\xi}}-h_{\xi}h_{\ov{\xi}} =\alpha\beta
(ts+jh-th-js)\xi^{j+t-1}\ov{\xi}^{h+s-1}+O(\xi^{j+t}).
\end{split}\end{equation*}
On the other hand, $j+h=t+s$, $j<t$. Thus $j\neq t$ and $h\neq s$.
Hence $ts+jh-th-js=(j-t)(h-s)\neq 0$. Thus
$hh_{\xi\ov{\xi}}-h_{\xi}h_{\ov{\xi}}$ is not identically $0$, which
contradicts  our hypothesis in (\ref{heq}).
\end{proof}

\begin{lem}\label{lem4-tech}
  Let $f(z,\ov{z})$ be a  weighted   homogeneous polynomial (with any assigned positive weight on $z$) in $(z,\ov{z})\in \mathbb{C}^n\times \mathbb{C}^n$,
   which is holomorphic in its variable $z_j$ for each  $j\in [1, k]$ with $k\le n$.
    Assume that Re$f(z,\ov{z})$ is a
   plurisubharmonic function without non-trivial holomorphic  terms. Then $f(z,\ov{z})$ is independent of $z_1,\cdots, z_k$ and
   $\ov{z_1},\cdots
   \ov{z_k}$.
\end{lem}

\begin{proof}
 We need only to prove the lemma with  $k=1$ and the other case follows from an induction argument.
 Since Re$f(z,\ov{z})$ is  plurisubharmonic, for each $j$ with $2\leq j\leq n$, we have
\begin{equation}\begin{split}\label{pshin}
(2\text{Re}f)_{z_1\ov{z_1}}(2\text{Re}f)_{z_j\ov{z_j}}-(2\text{Re}f)_{z_1\ov{z_j}}(2\text{Re}f)_{z_j\ov{z_1}}\geq
0.
\end{split}\end{equation}
Since $f(z,\ov{z})$  is holomorphic in $z_1$, we have
$$
(2\text{Re}f)_{z_1\ov{z_1}}=0,\
(2\text{Re}f)_{z_1\ov{z_j}}=f_{z_1\ov{z_j}},\
(2\text{Re}f)_{z_j\ov{z_1}}=\ov{f}_{z_j\ov{z_1}}.
$$

Substituting these relations back to (\ref{pshin}), we obtain
$-|f_{z_1\ov{z_j}}|^2\geq 0$. Thus $f_{z_1\ov{z_j}}\equiv0$. Since
$f(z,\ov{z})$ is holomorphic in $z_1$, we see that
$$g(z,\ov{z})=f(z,\ov{z})-f(0,z_2,\cdots,z_{n},0,\ov{z_2},\cdots,\ov{z_n})$$
is a holomorphic function. By our assumption,
$$Ref=Reg(z,\ov{z})+Ref(0,z_2,\cdots,z_{n},0,\ov{z_2},\cdots,\ov{z_n})$$
contains no non-trivial holomorphic terms. Hence $g(z,\ov{z})\equiv
0$, which implies that $f(z,\ov{z})$ is independent of $z_1$ and
$\ov{z_1}$.
\end{proof}

\begin{lem}\label{lem2-tech}
  Let $h(z,\ov{z})=\sum_{I\ov{J}}a_{I\ov{J}}z^I\ov{z^J}$ be a   real nonzero  plurisubharmonic  polynomial
  in $(z,\ov{z})\in \mathbb{C}^n\times
   \mathbb{C}^n,$
where $I=(i_1,\cdots,i_n), J=(j_1,\cdots,j_n)$ with  $i_l+j_l$ being
a fixed positive integer (independent of $I,J$) denoted by $k_l$ for
each $l\in[1, n]$. Assume that $h_{z_1\ov{z_1}}\not \equiv 0$. Then
    each $k_l$ is even and the coefficient of $\Pi_{l=1}^{n} |z_l|^{k_l}$ is positive.
\end{lem}

\begin{proof}
By the plurisubharmonicity  of $h(z,\ov{z})$, we know
$h_{z_1\ov{z_1}}\geq 0$. Since $h_{z_1\ov{z_1}}\not \equiv 0$, each
$k_j$ is even. Write $z_i=r_ie^{i\theta_i}$, then for any $R_i\in
(0,\infty)$, we have
\begin{equation}\begin{split}
&\frac{1}{(2\pi)^n}\int_0^{R_1}\cdots \int_0^{R_n}\cdot \int_0^{2\pi}\cdots \int_0^{2\pi} h_{z_1\ov{z_1}}dr_1\cdots dr_n d\theta_1 \cdots d\theta_n\\
=&\text{the coefficient of } \Pi |z_j|^{k_i} \cdot \text{some
positive constant} \geq 0.
\end{split}\end{equation}
If the coefficient of $\Pi_{j=1}^{n} |z_j|^{k_j}$ is $0$, then the
above integral is $0$. Combining with $h_{z_1\ov{z_1}}\geq 0$, we
obtain $h_{z_1\ov{z_1}}\equiv 0$. This contradicts our assumption
that $h_{z_1\ov{z_1}}\not \equiv 0. $ This proves Lemma
\ref{lem2-tech}.
\end{proof}

\begin{lem}\label{mainlem-tech}
  Let $B(z_1,\ov{z_1})$, $f(z_2,\ov{z_2})$ and $g(z_2,\ov{z_2})$ be three
   homogeneous polynomials of degree $k\ge 1$, $m\ge 1$ and $m\ge 1$, respectively, in the ordinary sense
    with $B(z_1,\ov{z_1})\not\equiv 0$, $f(z_2,\ov{z_2})\not\equiv 0$. Suppose that
  $B(z_1,0)=B(0,\ov{z_1})=0$. Suppose that $F=Bf+z_1^kg$ with Re$F$
  being
   a non zero plurisubharmonic polynomial
   without any non-trivial holomorphic term. Then $k$ and $m$ are even and
    Re$F=\alpha |z_1|^{k}|z_2|^m$ for some $\alpha> 0$.
\end{lem}

\begin{proof} By the assumption that Re$F$ is  non-zero and plurisubharmonic, $\left(Re(F)\right)_{z_1\ov{z_1}}\geq 0$.
Since $B(z_1,0)=B(0,\ov{z_1})=0$ and Re$F$ contains no non-trivial
holomorphic terms, one  further concludes that
$\left(Re(F)\right)_{z_1\ov{z_1}}$ is not identically $0$. By Lemma
\ref{lem2-tech}, $m$ and $k$ are even. Set $k=2k_3$ and $m=2m_3$.
Write
$$
B=\sum\limits_{j+h=k} B_{jh}z_1^j\ov{z_1}^h,\ f=\sum_{t+s=m}
f_{ts}z_2^t\ov{z_2}^s,\ g=\sum_{t+s=m} g_{ts}z_2^t\ov{z_2}^s.
$$

First we claim that $B_{k_3k_3}\neq 0$ and $f_{m_3m_3}\neq 0$.
Otherwise the coefficient of the $|z_1|^{2k_3-2}|z_2|^{2m_3}$ in
$\left(Re(F)\right)_{z_1\ov{z_1}}$ is zero, and thus by Lemma
\ref{lem2-tech}, we reach a contradiction.
After writing $F=cB\cdot \frac{1}{c}f+z_1^kg$, we  can assume that
$B_{k_3k_3}=1$.
\bigskip

By the plurisubharmonicity of Re$(F)$, we have
\begin{equation}\begin{split}
(\text{Re}F)_{z_1\ov{z_1}}(\text{Re}F)_{z_2\ov{z_2}}-(\text{Re}F)_{z_1\ov{z_2}}(\text{Re}F)_{z_2\ov{z_1}}\geq
0.
\end{split}\end{equation}

Notice that
\begin{equation}\begin{split}
2(\text{Re}F)_{z_1\ov{z_1}}=B_{z_1\ov{z_1}}f+\ov{B}_{z_1\ov{z_1}}\ov{f},\
2(\text{Re}F)_{z_2\ov{z_2}}=Bf_{z_2\ov{z_2}}+\ov{B}\ov{f}_{z_2\ov{z_2}}+2\text{Re}(z_1^kg_{z_2\ov{z_2}}).
\end{split}\end{equation}
Thus
\begin{equation}\begin{split}
4(\text{Re}F)_{z_1\ov{z_1}}(\text{Re}F)_{z_2\ov{z_2}}
=2\text{Re}\Big(BB_{z_1\ov{z_1}}ff_{z_2\ov{z_2}}+\ov{B}B_{z_1\ov{z_1}}f\ov{f}_{z_2\ov{z_2}}
+B_{z_1\ov{z_1}}f\cdot2\text{Re}(z_1^kg_{z_2\ov{z_2}}) \Big).
\end{split}\end{equation}

The coefficients of $|z_1|^{2k-2}$  in $BB_{z_1\ov{z_1}}$ and
$\ov{B}B_{z_1\ov{z_1}}$ are, respectively,
$$
\sum\limits_{j+h=k}jhB_{jh}B_{hj},\ \sum\limits_{j+h=k}jh|B_{hj}|^2.
$$

The coefficients of $|z_2|^{2m-2}$  in $ff_{z_2\ov{z_2}}$ and
$f\ov{f}_{z_2\ov{z_2}}$  are, respectively,
$$
\sum\limits_{t+s=m}tsf_{ts}f_{st},\ \sum\limits_{t+s=m}ts|f_{ts}|^2
$$

Notice that $B_{z_1\ov{z_1}}f\cdot\text{Re}(z_1^kg_{z_2\ov{z_2}})$
is not divisible by $|z_1|^{2k-2}$ (unless is is identically zero).
Hence the coefficient of $|z_1|^{2k-2}|z_2|^{2m-2}$  in
$4(\text{Re}f)_{z_1\ov{z_1}}(\text{Re}f)_{z_2\ov{z_2}}$ is
\begin{equation}\begin{split}
\sum\limits_{j+h=k,t+s=m}2\text{Re}\Big(jhB_{jh}B_{hj}tsf_{ts}f_{st}+jh|B_{hj}|^2ts|f_{ts}|^2\Big).
\end{split}\end{equation}

We similarly compute the coefficient of $|z_1|^{2k-2}|z_2|^{2m-2}$
in $4(\text{Re}F)_{z_1\ov{z_2}}(\text{Re}F)_{z_2\ov{z_1}}$ as
follows:
\begin{equation}\begin{split}
2(\text{Re}F)_{z_1\ov{z_2}}=&B_{z_1}f_{\ov{z_2}}+\ov{B}_{z_1}\ov{f}_{\ov{z_2}}+kz_1^{k-1}g_{\ov{z_2}},\\
2(\text{Re}F)_{z_2\ov{z_1}}=&B_{\ov{z_1}}f_{{z_2}}+\ov{B}_{\ov{z_1}}\ov{f}_{{z_2}}+k\ov{z_1}^{k-1}\ov{g}_{z_2}.
\end{split}\end{equation}
Thus
\begin{equation}\begin{split}
4(\text{Re}F)_{z_1\ov{z_2}}(\text{Re}F)_{z_2\ov{z_1}}
=&B_{z_1}B_{\ov{z_1}}f_{z_2}f_{\ov{z_2}}+B_{z_1}\ov{B}_{\ov{z_1}}\ov{f}_{z_2}f_{\ov{z_2}}
+B_{\ov{z_1}}\ov{B}_{z_1}f_{z_2}\ov{f}_{\ov{z_2}}+\ov{B}_{z_1}\ov{B}_{\ov{z_1}}\ov{f}_{\ov{z_2}}\ov{f}_{z_2}\\
&+2\text{Re}\Big(kz_1^{k-1}g_{\ov{z_2}}(B_{\ov{z_1}}f_{{z_2}}+\ov{B}_{\ov{z_1}}\ov{f}_{{z_2}})\Big)+k^2|z_1|^{2k-2}|g_{\ov{z_2}}|^2.
\end{split}\end{equation}

The coefficients of $|z_1|^{2k-2}$ in $B_{z_1}B_{\ov{z_1}}$,
$B_{z_1}\ov{B}_{\ov{z_1}}$, $B_{\ov{z_1}}\ov{B}_{{z_1}}$ and
$\ov{B}_{z_1}\ov{B}_{\ov{z_1}}$ are, respectively
$$
\sum\limits_{j+h=k}h^2B_{hj}B_{jh}, \
\sum\limits_{j+h=k}h^2|B_{hj}|^2,\ \sum\limits_{j+h=k}j^2|B_{hj}|^2,
\ \sum\limits_{j+h=k}j^2\ov{B_{hj}}\ov{B_{jh}}.
$$

The coefficients of $|z_2|^{2m-2}$  in $f_{\ov{z_2}}f_{z_2}$,
$f_{\ov{z_2}}\ov{f}_{z_2}$, $f_{{z_2}}\ov{f}_{\ov{z_2}}$ and
$\ov{f}_{\ov{z_2}}\ov{f}_{z_2}$ are, respectively,
$$
\sum\limits_{t+s=m}s^2f_{ts}f_{st},\
\sum\limits_{t+s=m}s^2|f_{ts}|^2,\
\sum\limits_{t+s=m}t^2|f_{ts}|^2,\
\sum\limits_{t+s=m}t^2\ov{f_{ts}}\ov{f_{st}}.
$$

Notice that
$kz_1^{k-1}g_{\ov{z_2}}(B_{\ov{z_1}}f_{{z_2}}+\ov{B}_{\ov{z_1}}\ov{f}_{{z_2}})$
is not divisible by $|z_1|^{2k-2}$ (when not identically zero).
Hence the coefficient of $|z_1|^{2k-2}|z_2|^{2m-2}$  in
$4(\text{Re}F)_{z_1\ov{z_2}}(\text{Re}F)_{z_2\ov{z_1}}$ is
\begin{equation*}\begin{split}
\sum\limits_{j+h=k,t+s=m}\Big(&h^2B_{jh}B_{hj}s^2f_{ts}f_{st}+h^2|B_{hj}|^2s^2|f_{ts}|^2+
j^2|B_{hj}|^2t^2|f_{ts}|^2\\
&+j^2\ov{B}_{jh}\ov{B}_{hj}t^2\ov{f}_{ts}\ov{f}_{st}\Big)+\sum\limits_{t+s=m}
k^2s^2|g_{ts}|^2.
\end{split}\end{equation*}

Hence the coefficient of $|z_1|^{2k-2}|z_2|^{2m-2}$  in
$4(\text{Re}F)_{z_1\ov{z_1}}(\text{Re}F)_{z_2\ov{z_2}}-4(\text{Re}F)_{z_1\ov{z_2}}(\text{Re}F)_{z_2\ov{z_1}}$
is
\begin{equation*}\begin{split}
&\sum\limits_{j+h=k,t+s=m}\Big\{2\text{Re}\big(jhB_{jh}B_{hj}tsf_{ts}f_{st}+jh|B_{hj}|^2ts|f_{ts}|^2\big)
-\big(h^2B_{hj}B_{jh}s^2f_{ts}f_{st}\\
&+h^2|B_{hj}|^2s^2|f_{ts}|^2+j^2|B_{hj}|^2t^2|f_{ts}|^2+j^2\ov{B}_{jh}\ov{B}_{hj}t^2\ov{f}_{ts}\ov{f}_{st}\big)\Big\}
-\sum\limits_{t+s=m} k^2s^2|g_{ts}|^2\\
=&-\sum\limits_{j+h=k,t+s=m}\Big\{(hs-jt)^2|B_{hj}|^2|f_{ts}|^2+hs(hs-jt)B_{jh}B_{hj}f_{ts}f_{st}\\
&+jt(jt-hs)\ov{B}_{jh}\ov{B}_{hj}\ov{f}_{ts}\ov{f}_{st}\Big\}-\sum\limits_{t+s=m} k^2s^2|g_{ts}|^2\\
=&-\sum\limits_{h\leq j, t\leq
s}\Gamma_{hj}^{ts}\Big\{(hs-jt)^2|B_{hj}|^2|f_{ts}|^2+(js-ht)^2|B_{jh}|^2|f_{ts}|^2
+(ht-js)^2|B_{hj}|^2|f_{st}|^2\\&
+(jt-hs)^2|B_{jh}|^2|f_{st}|^2+\Big(hs(hs-jt)+js(js-ht)+ht(ht-js)+jt(jt-hs)\Big)B_{jh}B_{hj}f_{ts}f_{st}\\
&+\Big(jt(jt-hs)+ht(ht-js)+js(js-ht)+hs(hs-jt)\Big)\ov{B}_{jh}\ov{B}_{hj}\ov{f}_{ts}\ov{f}_{st}\Big\}
-\sum_{t+s=m}  k^2s^2|g_{ts}|^2\\
=&-\sum\limits_{h\leq j, t\leq
s}\Gamma_{hj}^{ts}\Big\{(hs-jt)^2|B_{hj}|^2|f_{ts}|^2+(js-ht)^2|B_{jh}|^2|f_{ts}|^2
+(ht-js)^2|B_{hj}|^2|f_{st}|^2\\&
+(jt-hs)^2|B_{jh}|^2|f_{st}|^2+\Big((hs-jt)^2+(js-ht)^2\Big)B_{jh}B_{hj}f_{ts}f_{st}\\
&+\Big((ht-js)^2+(jt-hs)^2\Big)
\ov{B}_{jh}\ov{B}_{hj}\ov{f}_{ts}\ov{f}_{st}\Big\}-\sum_{t+s=m}
k^2s^2|g_{ts}|^2.
\end{split}\end{equation*}
Here we have set
\begin{equation*}\begin{split}
\Gamma_{hj}^{ts}=\left\{\begin{array}{ll} 1 & h<j,\ t<s,\\
\frac{1}{2} & h=j,\ t<s\ \text{or}\ h<j,\ t=s,\\ 0 & h=j,\ t=s.
\end{array}\right.
\end{split}\end{equation*}

 Notice, by the H\"older inequality, that
\begin{equation}\begin{split}
&|\Big((js-ht)^2+(hs-jt)^2\Big)B_{jh}B_{hj}f_{ts}f_{st}+\Big((ht-js)^2+(jt-hs)^2\Big)
\ov{B}_{jh}\ov{B}_{hj}\ov{f}_{ts}\ov{f}_{st}|\\
&\leq
(js-ht)^2(|B_{hj}f_{st}|^2+|B_{jh}f_{ts}|^2)+(jt-hs)^2(|B_{hj}f_{ts}|^2+|B_{jh}f_{st}|^2).
\end{split}\end{equation}
Thus we obtain  the coefficient of  $|z_1|^{2k-2}|z_2|^{2m-2}$  in
$$4(\text{Re}F)_{z_1\ov{z_1}}(\text{Re}F)_{z_2\ov{z_2}}-4(\text{Re}F)_{z_1\ov{z_2}}(\text{Re}F)_{z_2\ov{z_1}}$$
is non positive. Furthermore, this coefficient is $0$ if and only if
for $h\leq j$, $t\leq s$ and for any $j^*+l^*=m-1$ with $l^*\not
=0$:
\begin{equation}\begin{split}\label{cond}
B_{hj}f_{st}=-\ov{B_{jh}f_{ts}}\ \text{for}\ js\neq ht,\
B_{jh}f_{ts}=-\ov{B_{hj}f_{st}}\ \text{for}\ jt\neq hs,\
g_{j^*l^*}=0.
\end{split}\end{equation}

Repeating the argument in the proof of Lemma \ref{lem2-tech}, we
conclude that (\ref{cond}) holds and moreover
\begin{equation}\begin{split}\label{je0}
(\text{Re}F)_{z_1\ov{z_1}}(\text{Re}F)_{z_2\ov{z_2}}-(\text{Re}F)_{z_1\ov{z_2}}(\text{Re}F)_{z_2\ov{z_1}}=0.
\end{split}\end{equation}
Since $ReF$ and $Bf$  contain no non-trivial holomorphic terms,
we see $g\equiv 0$. Setting $j=h=k_3$ in (\ref{cond}) and using the
normalization that $B_{k_3k_3}=1$, we obtain $f_{ts}=-\ov{f_{st}}$
for $t\neq s$.


Now, if $f$ is of the form $f=f_{m_3m_3} |z_2|^m$ and
Re$F=|z_2|^mp(z_1,\ov{z_1})$, then (\ref{je0}) is equivalent to
$$
pp_{z_1\ov{z_1}}-p_{z_1}p_{\ov{z_1}}=0.
$$
By Lemma \ref{lem1-teh}, $p$ is a monomial. On the other hand, since
$p$ is real valued, $p=\alpha |z_1|^k$ for some $\alpha>0$. Namely,
Re$F=\alpha |z_1|^k|z_2|^m$. This proves the lemma.

For the rest of the proof, we  suppose that $f$ is not of the form
$f=f_{m_3m_3} |z_2|^m$. Since $f_{m_3m_3}\neq 0$,  $f$ is not a
monomial.

Now suppose
$$
\text{Re}F=z_1^h\ov{z_1}^jq(z_2,\ov{z_2})+O(z_1^{h+1}), \ q\not = 0.
$$
Since $B(z_1,0)=B(0,z_1)=0$, we have $h,j\geq 1$.  From (\ref{je0}),
we get
$$
hjz_1^{2h-1}\ov{z_1}^{2j-1}(qq_{z_2\ov{z_2}}-q_{z_2}q_{\ov{z_2}})+O(z_1^{2h})=0.
$$
This gives
$$
qq_{z_2\ov{z_2}}-q_{z_2}q_{\ov{z_2}}=0,
$$
which further forces  $q$ to be a monomial.

\bigskip

(1) If $B_{hj}=0$ or $B_{jh}=0$, then
$q=\frac{1}{2}\ov{B_{jh}}\ov{f}$ or $q=\frac{1}{2}B_{hj}{f}$,
respectively. In either case, $q$ is not a monomial and thus we
reach a contradiction.
\medskip

(2) Assume that $B_{hj}\neq 0$, $B_{jh}\neq 0$ and  $h< j$. In this
case, $B_{hj}f_{m_3m_3}=-\ov{B_{jh}f_{m_3m_3}}$. Hence there is no
term of the form $\gamma|z_2|^m$ in $q$. Setting $h=j$ in
(\ref{cond}), we see $f_{ts}=-\ov{f_{st}}$ for $t\neq s$. Thus
$B_{hj}f_{m_3m_3}|z_2|^m+\ov{B_{jh}f_{m_3m_3}|z_2|^m}=0$ and
$f-f_{m_3m_3} |z_2|^m=-\ov{(f-f_{m_3m_3} |z_2|^m)}$.
 Hence Re$F$ can be computed as follows:
\begin{equation}\begin{split}
Re(F)=&\frac{1}{2}(B_{hj}f+\ov{B_{jh}}\ov{f})z_1^h\ov{z_1}^j+O(z_1^{h+1})\\
=&\frac{1}{2}(B_{hj}-\ov{B_{jh}})z_1^h\ov{z_1}^j\cdot (f-f_{m_3m_3}
|z_2|^m)+O(z_1^{h+1}).
\end{split}\end{equation}

Thus we conclude that
$q=\frac{1}{2}(B_{hj}-\ov{B_{jh}})(f-f_{m_3m_3} |z_2|^m)$, which can
not be a monomial and thus gives a contradiction.

Hence we must have $h\geq j$. But from the reality of Re$F$, we see
that $B=|z_1|^k$ and  Re$F$ takes the form
$\frac{1}{2}|z_1|^k(f(z_2,\ov{z_2})+\ov{f(z_2,\ov{z_2})})$. Since
$f_{ts}=-\ov{f_{st}}$ for $s\neq t$,
we conclude that Re$F$ takes the form $\alpha |z_1|^k|z_2|^m$
with $\alpha>0$. This finally completes the proof of the lemma.
\end{proof}

\section{Proof of Theorem \ref{main-tech}}

\noindent  In this section, we provide a detailed proof of Theorem
\ref{main-tech}, which played a key role in the proof of our main
theorem. We write $z=(z_1,z_2)$ for the coordinates in ${\mathbb
C}^2$ in this section.

\begin{thm}\label{main-tech}
  Define the weight of $z_1$ and $\ov{z_1}$ to be $1$,  the weight of $z_2$ and $\ov{z_2}$ to be $k\in {\mathbb N}$ with $k>1$.
  Let  $A=A(z_1,\ov{z_1})$ be a homogenous polynomial of degree $k-1$ in $(z_1,\ov{z_1})$ without
   holomorphic terms. Suppose
  that $f$ is a weighted homogeneous polynomial in $(z,\ov{z})$ of
  weighted degree $m>k$. Further assume that  $\hbox{Re}(f)$ is plurisubharmonic, contains no
  non-trivial holomorphic terms and assume that $f$  satisfies the following equation:
\begin{equation}\begin{split}\label{be}
f_{\ov{z_1}}(z,\ov{z})+\ov{A(z_1,\ov{z_1})}f_{\ov{z_2}}(z,\ov{z})=0.
\end{split}\end{equation}
  Then Re$(f)\equiv0$.

\end{thm}


Without the plurisubharmonicity   on $\hbox{Re}(f)$, the above
theorem can not be true as the following simple example
demonstrates:

\begin{example}
Let  $L=\frac{\p}{\p z_1}-|z_1|^2\frac{\p}{\p z_2}$
and let  $f=z_1\ov{z_2}+\frac{1}{2}|z_1|^4$. Then $\ov{L}(f)\equiv
0$. Notice that $\hbox{Re}(f)$ is not plurisubharmonic neither is
$0$. Notice that $A=\ov{A}=-|z_1|^2, \hbox{Re}(f)$ has no
holomorphic terms. We also mention that in Theorem \ref{main-tech},
we can not conclude $f\equiv 0$ as demenstrated by the following
example: Let $L=\frac{\p}{\p z_1}+kz_1^k\ov{z^{k-1}_1}\frac{\p}{\p
z_2}$ and $f=i(z_2+\ov{z_2}-|z_1|^{2k})^2$. Then $\ov{L}f=0$ and
$\hbox{Re}(f)\equiv 0$. However $f\not\equiv 0$.
\end{example}

 \begin{proof}[Proof of Theorem \ref{main-tech}]
  The proof of Theorem \ref{main-tech} is  long.  We
 will proceed
according to the  four different  scenarios, two of which are
reduced to CR equations along finite type hypersurfaces where
Proposition \ref{disc} can be applied.

For $0\leq j\leq [\frac{m}{k}]:=m_0$, denoted by $f^{[j]}$ the sum
of terms  (monomial terms) in $f$ which has ordinary  degree $j$ in
$z_2$ and $\ov{z_2}$. Then
\begin{equation*}\begin{split}
f=f^{[m_0]}+f^{[m_0-1]}+\cdots+f^{[0]}.
\end{split}\end{equation*}
In the course of the proof, for $j=1,2$, we write $O(|z_j|^k)$ for a
homogeneous polynomial with (the ordinary or un-weighted) degree in
$z_j$ and $\ov{z_j}$  at least $k$.  We also denote by $L(|z_j|^k)$
a homogeneous polynomial with the un-weighted degree in $z_j$ and
$\ov{z_j}$  at most $k$. For a homogeneous polynomial
$P=\sum_{h+j=l} C_{hj}z_1^h\ov{z_1}^j$, we denote the integral of
$P$ along $\ov{z_1}$  as
\begin{equation}
F({P})=\sum_{h+j=l}\frac{1}{j+1}{C_{hj}}z_1^h\ov{z_1}^{j+1}.
\end{equation}

We remark that after a  transformation of the form:
$(z_1,z_2)\rightarrow (z_1,{\delta }^{-1}z_2)$, $A$ and $f$, in the
new coordinates still denoted by $(z_1,z_2)$, takes the form
\begin{equation}\label{afc}
\delta^{-1}A\ \text{and}\ f(z_1,\delta
z_2,\ov{z_1},\ov{\delta}\ov{z_2}).
\end{equation}
We will need this transformation to normalize certain coefficients
in our proof.

\medskip

{\bf Case I:} In this case, we suppose $km_0<m$ or $km_0=m$,
$f^{[m_0]}=0$.

\medskip
Suppose $h$ is the largest integer such that $f^{[h]}\neq0$. From
(\ref{be}), $f^{[h]}$ is holomorphic in $z_1$. We suppose that
\begin{equation}
f^{[h]}=z_1^j\sum\limits_{t+s=h} f_{ts}z_2^t\ov{z_2}^s,\ \
\text{here}\ j+kh=m.
\end{equation}
 We then have $j\geq 1$.
Since Re$f$ contains no holomorphic terms, $f_{h0}=0$. In
particular, we see that we must have $h\ge 1$. In what follows, we
set a term with a negative power to be zero.
\medskip

First, we claim  $f_{ts}= 0$ for any $t\geq 1$. Since Re$f$ is
plurisubharmonic, we obtain
$$
\big(\text{Re}(f)\big)_{z_2\ov{z_2}}=\big(\text{Re}(f^{[h]})\big)_{z_2\ov{z_2}}+L(|z_2|^{h-3})\geq
0.
$$
For $|z_2|\gg |z_1|$, we get
$$
z_1^j\sum\limits_{t+s=h,t,s\geq 1}
tsf_{ts}z_2^{t-1}\ov{z_2}^{s-1}+\ov{z_1}^j\sum\limits_{t+s=h,t,s\geq
1} ts\ov{f_{ts}}z_2^{s-1}\ov{z_2}^{t-1}\geq 0.
$$
Since $j\geq 1$, this is possible only when the left hand side is
identically $0$. This implies that $f_{ts}=0$ for all $t\geq 1$.
Thus
$$
f^{[h]}=f_{0h} z_1^j\ov{z_2}^h\ \text{with}\ j,h\geq 1,\ j+kh=m.
$$
Then
$$
\text{Re}(f^{[h]})=\frac{1}{2}f_{0h}z_1^j\ov{z_2}^h+\frac{1}{2}\ov{f_{0h}}{z_2}^h\ov{z_1}^j.
$$
Since Re$(f)$ is plurisubharmonic, we have
$$
\big(\text{Re}(f)\big)_{z_1\ov{z_1}}\big(\text{Re}(f)\big)_{z_2\ov{z_2}}-
\big(\text{Re}(f)\big)_{z_1\ov{z_2}}\big(\text{Re}(f)\big)_{z_2\ov{z_1}}\geq
0.
$$
Notice that
\begin{equation*}\begin{split}
&\big(\text{Re}(f)\big)_{z_1\ov{z_1}}=O(|z_1|^{j-1}),\ \big(\text{Re}(f)\big)_{z_2\ov{z_2}}=O(|z_1|^{j+1})\\
&\big(\text{Re}(f)\big)_{z_1\ov{z_2}}=\frac{1}{2}f_{0h}jhz_1^{j-1}\ov{z_2}^{h-1}+O(|z_1|^{j}),\\
&\big(\text{Re}(f)\big)_{z_2\ov{z_1}}=\frac{1}{2}\ov{f_{0h}}jhz_2^{h-1}\ov{z_1}^{j-1}+O(|z_1|^{j}).
\end{split}\end{equation*}
Hence
\begin{equation}\begin{split}
&\big(\text{Re}(f)\big)_{z_1\ov{z_1}}\big(\text{Re}(f)\big)_{z_2\ov{z_2}}-
\big(\text{Re}(f)\big)_{z_1\ov{z_2}}\big(\text{Re}(f)\big)_{z_2\ov{z_1}}\\
=&-\frac{1}{4}j^2h^2|f_{0h}|^2|z_1|^{2j-2}|z_2|^{2h-2}+O(|z_1|^{2j-1})\geq
0.
\end{split}\end{equation}
Again, by choosing $|z_2|\gg |z_1|$, we get
$-j^2h^2|f_{0h}|^2|z_1|^{2j-2}|z_2|^{2h-2}\geq 0$. Hence $f_{0h}=0$,
which means that $f^{[h]}\equiv 0$. This contradicts  our assumption
that $f^{[h]}\neq 0$. Hence  { Case I} can only occur when
Re$(f)=0$.

\medskip

{\bf Case II:} \ We now assume that $km_0=m$, Re$(f^{[m_0]})\neq 0$.

\medskip
Suppose
$$
f^{[m_0]}=\sum\limits_{t+s=m_0} f_{ts}z_2^t\ov{z_2}^s.
$$

Since Re$(f^{[m_0]})$ contains no non-trivial holomorphic  terms, we
have $f_{0m_0}=-\ov{f_{m_00}}$. By the plurisubharmonicity  of
Re($f$), we get $($Re$(f^{[m_0]}))_{z_2\ov{z_2}}\geq 0$ and can not
be identically zero. By Lemma \ref{lem2-tech}, $m_0$ is even and
Re$f_{m_1m_1}> 0$. Here $m_0=2m_1$.

After a  rotation transformation of the form $(z_1,z_2)\rightarrow
(z_1,\delta^{-1} z_2)$ for some constant $\delta\not =0$, by
(\ref{afc}), we can make
\begin{equation}\label{tt2}
f_{(m_1-1)(m_1+1)}=c\ov{f_{m_1m_1}}\ \text{for a certain}\ c\geq 0.
\end{equation}
We remark that this transformation does not change our original
hypotheses in this case. Now (\ref{be}) can be solved as
\begin{equation}\label{solution}
f=-F(\ov{A})f_{\ov{z_2}}+\sum_{j=0}^{m}z_1^{jk}h^{[m-j]}(z_2,\ov{z_2}),
\ h^{[m-j]}(z_2,0)=0,\ \hbox{for each } j.
\end{equation} In particular, we get
$$f^{[m_0-1]}=-F(\ov{A})\cdot
f^{[m_0]}_{\ov{z_2}}+z_1^kg^{[m_0-1]}(z_2,\ov{z_2}),\ g(z_2,0)=0.
$$

By the plurisubharmonicty  of Re$f$, we have $
\text{Re}(f)_{z_1\ov{z_1}}\geq 0. $ Notice that $F(\ov{A})$ is
divisible by $|z_1|^2$. Hence
$$
\text{Re}(f)_{z_1\ov{z_1}}=(\text{Re}f^{[m_0-1]})_{z_1\ov{z_1}}+L(|z_2|^{m_0-2})\geq
0.
$$
Hence
$$
(\text{Re}f^{[m_0-1]})_{z_1\ov{z_1}}\geq 0
$$
Notice that the  (ordinary) degree of
$(\text{Re}f^{[m_0-1]})_{z_1\ov{z_1}}$ in $z_2$ and $\ov{z_2}$ is
$m_0-1$ which is an odd number, we  have
$$
(\text{Re}f^{[m_0-1]})_{z_1\ov{z_1}}\equiv 0.\ \text{Thus, it follws
that }\ \text{Re}(F(\ov{A})\cdot f^{[m_0]}_{\ov{z_2}})=0.
$$

Next, write $ \ov{A}=\sum_{j+h=k-1,h\geq
1}\ov{A_{jh}}{z_1^h}\ov{z_1}^j. $ Then
$$
F(\ov{A})=\sum\limits_{j+h=k-1,h\geq
1}\frac{1}{j+1}\ov{A_{jh}}{z_1^h}\ov{z_1}^{j+1}.
$$
Hence
\begin{equation}\begin{split}
\text{Re}(F(\ov{A})\cdot
f^{[m_0]}_{\ov{z_2}})=\text{Re}\Big(\sum\limits_{j+h=k-1,h\geq
1}\frac{1}{j+1}\ov{A_{jh}}{z_1^h}\ov{z_1}^{j+1}\cdot
\sum\limits_{t+s=m_0,s\geq 1}sf_{ts}{z_2^t}\ov{z_2}^{s-1}\Big)=0.
\end{split}\end{equation}
Hence for $h+j=k$, $t+s=m_0-1$, we have
\begin{equation}\begin{split}
\frac{1}{j}\ov{A_{(j-1)h}}\cdot
(s+1)f_{t(s+1)}=-\frac{1}{h}\ov{\ov{A_{(h-1)j}}\cdot
(t+1)f_{s(t+1)}}.
\end{split}\end{equation}
Setting $t=m_1-1$, $s=m_1$ in the above equation and making use of
(\ref{tt2}), we get
\begin{equation}\begin{split}\label{Ajhhj}
\frac{1}{j}\ov{A_{(j-1)h}}\cdot
(m_1+1)c=-\frac{1}{h}\ov{\ov{A_{(h-1)j}}\cdot m_1}.
\end{split}\end{equation}
If $c=0$, then $A_{(h-1)j}=0$ for all $h+j=k,\ h\geq 1, \ j\geq 1$.
This implies that $A\equiv0$, which is impossible. Thus $c\neq 0$.
From (\ref{Ajhhj}), we  get
\begin{equation}\label{alpha}
{A_{(j-1)h}}{A_{(h-1)j}}\leq 0\ \text{and the equality holds only
when}\ A_{(j-1)h}=A_{(h-1)j}=0.
\end{equation}

Next, by (\ref{be}) (\ref{solution}), we compute the following:
$$
f^{[m_0-2]}=F(\ov{A}F(\ov{A}))\cdot f^{[m_0]}_{\ov{z_2}^2}-F(\ov{A}
z_1^k)g^{[m_0-1]}_{\ov{z_2}}+z_1^{2k}g^{[m_0-2]}(z_2,\ov{z_2}).
$$
We will compute the coefficient of  $|z_1|^{2k}|z_2|^{m_0-2}$ in
$f^{[m_0-2]}$. First, the coefficient of $|z_2|^{m_0-2}$ in
$f^{[m_0]}_{\ov{z_2}^2}$ is $ (m_1+1)m_1 f_{(m_1-1)(m_1+1)}. $
Notice that
\begin{equation}\begin{split}
\ov{A}F(\ov{A})=\sum\limits_{j+h=k-1}\ov{A_{jh}}z_1^h\ov{z_1}^j\cdot
\sum\limits_{t+s=k-1}\frac{1}{t+1}\ov{A_{ts}}z_1^s\ov{z_1}^{t+1}
\end{split}\end{equation}
Hence
\begin{equation}\begin{split}
F(\ov{A}F(\ov{A}))=\sum\limits_{j+h=k-1,t+s=k-1}\frac{1}{(t+1)(j+t+2)}\ov{A_{jh}}\ov{A_{ts}}z_1^{h+s}\ov{z_1}^{j+t+2}
\end{split}\end{equation}
When
$
h+s=j+t+2,\ j+h=k-1,\ t+s=k-1,
$
we have
$
j=k-1-h,\ t=h-1,\ s=k-h.
$
Hence the coefficient  in $F(\ov{A}F(\ov{A}))$  with the factor
$|z_1|^{2k}$ is
$$
\sum\limits_{1\leq h\leq
k-1}\frac{1}{hk}\ov{A_{(k-1-h)h}}\ov{A_{(h-1)(k-h)}}:=H.
$$
By (\ref{alpha}), $H\leq 0$, also $H=0$ if and if $A_{(k-1-h)h}=0$
for all $h\geq 1$. This is equivalent to $A=0$, which is impossible.
Thus $H<0$.

Notice that $\ov{A}$ is divisible by $z_1$, thus $F(\ov{A} z_1^k)$
does not contain $|z_1|^{2k}$ term.

Thus the coefficient of  $|z_1|^{2k}|z_2|^{m_0-2}$ in $f^{[m_0-2]}$
is $ (m_1+1)m_1f_{(m_1-1)(m_1+1)}H. $ Recall that Re$f_{m_1m_1}>0$
and $c>0$.  Together with (\ref{tt2}), we get
Re$f_{(m_1-1)(m_1+1)}>0$. Hence the real part of the coefficient of
 $|z_1|^{2k}|z_2|^{m_0-2}$  in $f^{[m_0-2]}$ must be negative.
This contradicts  the following
$$
\Big(\text{Re}(f^{[m_0-2]})\Big)_{z_1\ov{z_1}}\geq 0,
$$
which is true due to the fact that
$\Big(\text{Re}(f^{[m_0]})\Big)_{z_1\ov{z_1}}\ \hbox{and} \
\Big(\text{Re}(f^{[m_0-1]})\Big)_{z_1\ov{z_1}}=0.$
\medskip

{\bf Case III:}  $m=km_0$, $f^{[m_0]}\neq 0$, Re$(f^{[m_0]})=0$ and
Re$(f^{[m_0-1]})\neq 0$.

\medskip

Here, we reduce   $f$ to the solution of a CR vector field of a real
hypersurface of finite type in ${\mathbb C}^2$  and then apply
Proposition \ref{disc} to reach a contradiction. Write
$B:=-F(\ov{A})=\sum_{j+h=k} B_{jh}z_1^j\ov{z_1}^h$. By Lemma
\ref{lem2-tech}, both $k$ and $m_0-1$ are even. Define $k=2k_2$,
$m_0=2m_2+1$. Then $B_{k_2k_2}\neq 0$ by Lemma \ref{lem2-tech},
which implies that $A_{(k_2-1)k_2}\neq 0$. After a dilation
transform of the form as in (\ref{afc}), we  assume that
$A_{(k_2-1)k_2}=-k_2$. Then $B_{k_2k_2}=1$. A direct computation
shows
$$
f^{[m_0-1]}=Bf^{[m_0]}_{\ov{z_2}}+z_1^kg(z_2,\ov{z_2}).
$$
From our assumption, Re$f^{[m_0-1]}$ is plurisubharmonic. By Lemma
\ref{mainlem-tech},
\begin{equation}\begin{split}\label{fm1e}
g(z_2,\ov{z_2})=0,\
\text{Re}(f^{[m_0-1]})=\text{Re}(Bf^{[m_0]}_{\ov{z_2}})=\lambda
|z_1|^k|z_2|^{m_0-1},\ \lambda > 0.
\end{split}\end{equation}
Notice that  $B_{k_2k_2}=1$ and
$(m_2+1)\hbox{Re}(f_{m_2(m_2+1)})=\lambda\neq 0$. Since
$\hbox{Re}(f^{[m_0]})=0$, we have
$f_{(m_2+1)m_2}+\ov{f_{m_2(m_2+1)}}=0$. Notice that
$f^{[m_0]}_{\ov{z_2}}-(m_2+1)f_{m_2(m_2+1)}|z_2|^{2m_2}$  has no
term divisible by $|z_2|^{2m_2}$. Hence we conclude from
(\ref{fm1e})
$$\hbox{Re}(Bf_{m_2(m_2+1)}(m_2+1))|z_2|^{m_0-1}= \lambda
|z_1|^k|z_2|^{m_0-1}.$$
Collecting terms  divisible by $z_2^{m_2+1}\ov{z_2}^{m_2-1}$ in
(\ref{fm1e}), we get
 $$m_2 B {f_{(m_2+1)m_2}}+\ov{B (m_2+2)
f_{(m_2-1)(m_2+2)}}=0.$$ Hence $B$ is different from $\ov{B}$ by a
constant. Since we normalized $B_{k_2k_2}=1$, we see that $B$ is
real-valued. But $f^{[m_0]}_{\ov{z_2}}$ contains a term of the form
$\mu |z_2|^{m_0-1}$ with Re$\mu\neq 0$. Thus $-F(\ov{A})=|z_1|^k$,
namely, $A=-k_2z_1^{k_2-1}\ov{z_1}^{k_2}$.

Now, $L=\frac{\partial}{\partial
z_1}+A(z_1,\ov{z_1})\frac{\partial}{\partial z_2}$ forms a basis for
the sections of CR vector fields  along  the real algebraic finite
type hypersurface $M_0$ in ${\mathbb C}^2$ defined by
$-z_2-\ov{z_2}=|z_1|^{2k_2} $ and $\ov{L}(f)\equiv 0.$ Thus $f$ is a
CR polynomial on $M_0$ and $g=f(z_1,\ov{z_1},z_2,
-z_2-|z_1|^{2k_2})$ is a weighted homogeneous holomorphic polynomial
of degree $m>k$.  Since $f-g\equiv 0$ over $M_0$, $M_0$ is contained
in the zero set of the plurisubharmonic $\rho=\hbox{Re}(f-g)$ with
$0\in M_0$. Notice that $\rho=O(|z|^2)$, we conclude by Proposition
\ref{disc} that $\rho\equiv 0$ or $\hbox{Re}(f)$ is  pluriharmonic.
This is a contradiction.
 Hence Case III cannot occur.

\medskip
{\bf Case IV:}   $m=km_0$, $f^{[m_0]}\neq 0$ but
Re$(f^{[m_0]})=$Re$(f^{[m_0-1]})= 0$.
\medskip

Write
$$
A=\sum_{h+j=k-1} A_{hj}z^h\ov{z}^j,\ B=-F(\ov{A})=\sum_{h+j=k}
B_{hj}z^h\ov{z}^j,\ f^{[m_0]}=\sum_{t+s=m_0} f_{ts}z_2^t\ov{z_2}^s.
$$
Then by our assumption that Re$(f^{[m_0]})=$Re$(f^{[m_0-1]})= 0$ and
$F(\ov{A})$ is divisible by $|z_1|^2$, we obtain as in Case (II) the
following
\begin{equation}\label{tre}
f^{[m_0-1]}=-F(\ov{A})\cdot f^{[m_0]}_{\ov{z_2}},\
f_{ts}=-\ov{f_{st}},\
B_{hj}(s+1)f_{t(s+1)}=-(t+1)\ov{B_{jh}f_{s(t+1)}}.
\end{equation}
Hence for each pair $(h,j)$, if $B_{hj}\neq 0$, then $B_{jh}\neq 0$;
for  otherwise we get $f_{ts}=0$ for any $t+s=m_0$ and reach a
contradiction. Since $B$ is nonzero, we can suppose there is a pair
$(h_0,j_0)$ such that $B_{h_0j_0}\neq 0$ and thus $B_{j_0h_0}\neq
0$. Since $f^{[m_0]}\neq 0$, there is a certain $f_{t_0(s_0+1)}\neq
0$ and thus $f_{s_0(t_0+1)}\neq 0$. By (\ref{tre}), we have
$$
B_{h_0j_0}(s_0+1)f_{t_0(s_0+1)}=-(t_0+1)\ov{B_{j_0h_0}f_{s_0(t_0+1)}},$$
$$B_{j_0h_0}(s_0+1)f_{t_0(s_0+1)}=-(t_0+1)\ov{B_{h_0j_0}f_{s_0(t_0+1)}}.
$$
Since $f_{t_0(s_0+1)}\neq 0$ and $f_{s_0(t_0+1)}\neq 0$, we have
$|B_{h_0j_0}|=|B_{j_0h_0}|$. After  a rotational transformation as
in (\ref{afc}) with a suitable choice of $\delta$, we can assume
that $B_{h_0j_0}=\ov{B_{j_0h_0}}$. Then by (\ref{tre}), we have
\begin{equation}\label{freal}
f_{ts}=-\ov{f_{st}},\ (s+1)f_{t(s+1)}=-(t+1)\ov{f_{s(t+1)}}.
\end{equation}
By (\ref{freal}), $f^{[m_0]}_{{\ov{z_2}}}$ is pure imaginary. Also,
it is not identically zero
for the absolute value of each coefficient is a non-zero multiple of
the others and at least one of them is non-zero. Now, by the first
equation in $(\ref{tre})$, we easily conclude
that $F(\ov{A})$ is a real-valued homogeneous polynomial  divisible
by $|z_1|^2$. Hence, $L=\frac{\partial}{\partial
z_1}+A(z_1,\ov{z_1})\frac{\partial}{\partial z_2}$ forms a basis for
the sections of CR vector fields  along  the real algebraic finite
type  hypersurface $M_0$ in ${\mathbb C}^2$ defined by
$z_2+\ov{z_2}=F(\ov{A})$  and $\ov{L}(f)\equiv 0.$ Now, following
the same argument as in Case (III), we achieve a contradiction
unless Re$(f)\equiv 0$.

\medskip
Combining our arguments in Cases I-IV, we conclude the proof of
Theorem \ref{main-tech}.
\end{proof}

\medskip \medskip

\bibliographystyle{amsalpha}

\noindent Xiaojun Huang, Department of Mathematics, Rutgers
University, New Brunswick, NJ 08903, USA
(huangx$@$math.rutgers.edu);

\noindent Wanke Yin,   School of Mathematics and Statistics, Wuhan University, Wuhan, Hubei 430072,
China (wankeyin@whu.edu.cn).

\end{document}